\documentclass[pdflatex,sn-mathphys-num]{sn-jnl}

\usepackage{graphicx}%
\usepackage{multirow}%
\usepackage{amsmath,amssymb,amsfonts}%
\usepackage{amsthm}%
\usepackage{mathrsfs}%
\usepackage[title]{appendix}%
\usepackage{xcolor}%
\usepackage{textcomp}%
\usepackage{manyfoot}%
\usepackage{booktabs}%
\usepackage{algorithm}%
\usepackage{algorithmicx}%
\usepackage{algpseudocode}%
\usepackage{listings}%

\theoremstyle{thmstyleone}%
\newtheorem{theorem}{Theorem}
\newtheorem{proposition}[theorem]{Proposition}%
\newtheorem{corollary}[theorem]{Corollary}
\newtheorem{lemma}[theorem]{Lemma}

\theoremstyle{thmstyletwo}%
\newtheorem{remark}{Remark}%

\theoremstyle{thmstylethree}%

\newcommand{\R}{\operatorname{\mathbb{R}}}

\newcommand{\C}{\operatorname{\mathbb{C}}}
\newcommand{\Z}{\operatorname{\mathbb{Z}}}
\newcommand{\re}{\operatorname{Re}}
\newcommand{\im}{\operatorname{Im}}

\newcommand{\z}{\operatorname{\zeta_K}}

\begin{document}

\title[Estimating the number of zeros of Dedekind zeta-functions]{Estimating the number of zeros of Dedekind zeta-functions}


\author{\fnm{Victor} \sur{Amberger}}

\affil{%
  \orgdiv{Department of Mathematical Sciences},
  \orgname{Chalmers University of Technology and University of Gothenburg},
  \city{Gothenburg},
  \postcode{41296},
  \country{Sweden}
}


\abstract{In this article, I develop a new approach to estimating the number of non-trivial zeros of a Dedekind zeta function with absolute height at most $T\ge1$, counted with multiplicity. The method yields an explicit zero-counting estimate with constants $(C_1,C_2,C_3)=(0.194,7.806,1.808)$. For the Riemann zeta function, it gives the error bound $0.097\log T+4.807$.}



\maketitle

\section{Introduction}

In this article, we consider Dedekind zeta functions associated with a given algebraic number field. Let $K$ be an algebraic number field of degree $n_K$. Then the associated Dedekind zeta-function is given by 
\begin{equation}\label{eq:dedekind-zeta}
    \z(s)=\sum_{I\subset\mathcal{O}_K\setminus0}N_K(I)^{-s},
\end{equation}
for $\re(s)>1$ \cite[Definition 5.1]{Neukirch1999}, where the sum is taken over all non-zero integral ideals of $K$. Just like the Riemann zeta function is related to the distribution of primes, the Dedekind zeta functions are closely related to the distribution of prime ideals. 
Furthermore, the Dedekind zeta function has properties very similar to the Riemann zeta function, in that the definition in \eqref{eq:dedekind-zeta} allows a meromorphic continuation to the entire complex plane which has a simple pole at $1$ and satisfies a functional equation. We have
\begin{equation}\label{eq:completed-zeta}
    \xi_K(s):=s(s-1)d_K^{s/2}\gamma_K(s)\z(s),
\end{equation}
where
\begin{equation}\label{eq:gamma-factor}
    \gamma_K(s)=\left(\pi^{-\frac{s+1}{2}}\Gamma\left(\frac{s+1}{2}\right)\right)^{r_2}\left(\pi^{-\frac{s}{2}}\Gamma\left(\frac{s}{2}\right)\right)^{r_1+r_2}.
\end{equation}
Here, $d_K$ denotes the absolute discriminant of the number field $K$. Furthermore, $r_1$ and $r_2$ denote the number of real and complex places of $K/\mathbb{Q}$. So, in particular, we have $n_K=r_1+2r_2$. Due to computational advantages, we will also use this representation for the gamma factor
\begin{equation}\label{eq:gamma-factor-legendre}
    \gamma_K(s)=\gamma_1(s)^{r_1}\gamma_2(s)^{r_2}
\end{equation}
with
\begin{equation*}
    \gamma_1(s):=\pi^{-\frac{s}{2}}\Gamma\left(\frac{s}{2}\right)\ \text{ and }\ \gamma_2(s):=\pi^{-s}2^{1-s}\Gamma(s),
\end{equation*}
which we may obtain from equation \eqref{eq:gamma-factor} by Legendre's duplication formula \cite[Corollary 12.1.3]{WhittakerWatson1920} for the Gamma function.
The completed Dedekind zeta function satisfies the functional equation \cite[Corollary 5.10]{Neukirch1999} as displayed below,
\begin{equation}\label{eq:functional-equation}
    \xi_K(s)=\xi_K(1-s).
\end{equation}
In this context, we study the distribution of the zeros of $\z(s)$. From the Euler product for the Dedekind zeta function $\z(s)$ \cite[Proposition 5.2]{Neukirch1999}
\begin{equation}\label{eq:euler-product}
    \z(s)=\prod_{\mathfrak{p}\subset\mathcal{O}_K}\left(1-N(\mathfrak{p})^{-s}\right)^{-1}
\end{equation}
we can see that $\z(s)$ and, by extension $\xi_K(s)$ have no zeros with $\re(s)>1$ and by the functional equation we can tell that all zeros of $\xi_K(s)$ must be located inside the critical strip with $\re(s)\in[0,1]$. We call these the non-trivial zeros of $\z(s)$ and throughout this paper we will denote them by $\rho=\beta+i\omega$, where $\beta$ denotes the real part of the zero and $\omega$ denotes the imaginary part. In this article, we investigate the vertical distribution of the non-trivial zeros.
Our goal is to estimate the quantity $N_K(T)$ given by
\begin{align*}
    N_K(T):=&\#\left\{\rho\in\C \mid \z(\rho)=0,\ 0<\beta<1,\ |\omega|\leq T\right\}'\\
    =&\ \ \ \frac{1}{2}\#\left\{\rho\in\C\mid \z(\rho)=0,\ 0<\beta<1,\ |\omega|\leq T\right\}\\
    &+\frac{1}{2}\#\left\{\rho\in\C\mid \z(\rho)=0,\ 0<\beta<1,\ |\omega|< T\right\}
\end{align*}
counted with multiplicity. Here, the prime notation indicates that $\z$-zeros with imaginary part $\pm T$ are counted with half their weight. Following and improving on the work of Backlund \cite{Backlund1916}, McCurley \cite{McCurley1984}, Rosser \cite{Rosser1941}, Kadiri and Ng \cite{KadiriNg2012}, Hasanalizade, Shen and Wong \cite{HasanalizadeShenWong2022Dedekind} were able to prove that
\begin{equation*}
    \left|N_K(T)-\frac{T}{\pi}\log\left(d_K\left(\frac{T}{2\pi e}\right)^{n_K}\right)\right|\leq C_1\left(\log d_K+n_K\log T\right)+C_2n_K+C_3
\end{equation*}
for $T\geq T_0$ and admissible tuples $(C_1,C_2,C_3,T_0)$. Table \ref{tab:previous-dedekind} below contains such admissible tuples.
\begin{table}[h!]
\centering
\begin{tabular}{|c|c|c|c|c|}
\hline
& \multicolumn{2}{c|}{$T \geq 1$} 
& \multicolumn{2}{c|}{$T \geq 10$} \\
\hline
$C_1$ & $C_2$ & $C_3$ & $C_2$ & $C_3$ \\
\hline
0.22737 & 23.02528 & 4.51954 & 22.97204 & 3.30668 \\
0.24493 & 6.66558  & 4.21201 & 6.60397  & 3.12362 \\
0.26304 & 5.22032  & 4.08149 & 5.15251  & 3.05074 \\
0.28032 & 4.43521  & 4.00936 & 4.36214  & 3.01124 \\
0.29590 & 3.93889  & 3.96852 & 3.86136  & 2.98903 \\
\hline
\end{tabular}
\caption{Admissible choices for $(C_1,C_2,C_3)$, taken from \cite{HasanalizadeShenWong2022Dedekind}}
\label{tab:previous-dedekind}
\end{table}\\
Similarly, more refined estimates exist for the Riemann zeta function. In this case, we define the zero counting function $N(T)$ as the number of zeros with real part between $0$ and $1$ and imaginary part between $0$ and $T$. We then have that
\begin{equation*}
    \left|N(T)-\frac{T}{2\pi}\log\left(\frac{T}{2\pi e}\right)\right|\leq C_1\log(T)+C_2\log\log(T)+C_3
\end{equation*}
for $T\geq T_0$. Table \ref{tab: previous results} summarizes the advances made for this estimate.
\begin{table}[h!]
\centering
\begin{tabular}{|l|c|c|c|c|}
\hline
 & $C_1$ & $C_2$ & $C_3$ & $T_0$ \\
\hline
von Mangoldt \cite{vonMangoldt1905} (1905) & 0.4320 & 1.9167 & 13.0788 & 28.5580 \\
Grossmann \cite{Grossmann1913} (1913) & 0.2907 & 1.7862 & 7.0120 & 50 \\
Backlund \cite{Backlund1916} (1916) & 0.1370 & 0.4430 & 5.2250 & 200 \\
Rosser \cite{Rosser1941} (1941) & 0.1370 & 0.4430 & 2.4630 & 2 \\
Trudgian \cite{Trudgian2012} (2012) & 0.1700 & 0 & 2.8730 & $e$ \\
Trudgian \cite{Trudgian2014} (2014) & 0.1120 & 0.2780 & 3.3850 & $e$ \\
Platt--Trudgian \cite{PlattTrudgian2015} (2015) & 0.1100 & 0.2900 & 3.165 & $e$ \\
Hasanalizade, Shen, and Wong \cite{HasanalizadeShenWong2022Riemann} (2022) & 0.1038 & 0.2573 & 9.4925 & $e$ \\
Bellotti--Wong \cite{BellottiWong2025}(2025) & 0.10076 & 0.24460 & 8.08344 & $e$ \\
\hline
\end{tabular}
\caption{Table with advances made for this estimate, taken from \cite{BellottiWong2025}}
\label{tab: previous results}
\end{table}
The approach of this article differs from the methods underlying the estimates above. Inspired by Turing's method \cite[Lemma 1]{Turing1953}, we shift part of the contour arising from the argument principle into the half-plane of absolute convergence. The remaining contribution can then be expressed as an absolutely convergent sum over the non-trivial zeros of $\z$. We bound the contribution of each zero by a linear combination of first- and second-order poles evaluated to the right of the critical strip. After summing over the zeros, these terms can be expressed in terms of logarithmic derivatives of the completed Dedekind zeta-function.\\
This formulation has the advantage that the resulting upper and lower bounds decompose according to the factors of the completed zeta-function. The contributions from the discriminant, the gamma factor, and the Dedekind zeta-function can therefore be estimated separately, with the latter evaluated entirely in the half-plane of absolute convergence. The quality of the leading coefficient is governed by an auxiliary optimization problem for the contribution of an individual zero. This reduces the problem of improving the leading constant to finding suitable parameters for this majorant. The particular choice used for our main result leads to the coefficient $0.194$, while other choices allow a trade-off between the leading coefficient and the lower-order terms.\\
Subsequent to an earlier version of this article, Bellotti and Wong \cite{BellottiWong2026Artin} extended the method developed here to Artin $L$-functions. Their refinement introduces an additional $a_4$-term in Lemma \ref{lem:f-bounds}, corresponding to a third logarithmic derivative, whereas the present article uses terms only up to the second logarithmic derivative. In particular, their refinement permits a reduction of the leading coefficient to approximately $0.1892$. I am also grateful to Bellotti and Wong for pointing out a sign inconsistency involving the $a_3$-term in an earlier version of this article, which has been corrected here.

\section{Results}

Inspired by Turing's method \cite[Lemma 1]{Turing1953}, we take a different approach to estimate this quantity, yielding the following result.
\begin{theorem}\label{thm:main}
    Let $K$ be a number field with degree $n_K$ and absolute discriminant $d_K$. Then
    \begin{equation*}
        \left|N_K(T)-\frac{T}{\pi}\log\left(d_K\left(\frac{T}{2\pi e}\right)^{n_K}\right)\right|\leq C_1\left(\log d_K+n_K\log T\right)+C_2n_K+C_3
    \end{equation*}
    for $T\geq1$, where $(C_1,C_2,C_3)=(0.194,7.806,1.808)$.
\end{theorem}
Further admissible tuples $(C_1,C_2,C_3)$ are given in table \ref{tab:dedekind-tuples}. As a corollary of Theorem \ref{thm:main} we obtain the following result.
\begin{corollary}\label{cor:riemann}
    Let $N(T)$  be defined by $\#\{\rho\in\C\mid\zeta(\rho)=0,\ \re(\rho)\in[0,1],\ \im(\rho)\in[0,T]\}'$ counted with multiplicity and such that zeros with imaginary part $T$ are counted with half weight. Then
    \begin{equation*}
        \left|N(T)-\frac{T}{2\pi}\log\left(\frac{T}{2\pi e}\right)\right|\leq C_1\log T+C_2
    \end{equation*}
    for $T\geq1$, where $(C_1,C_2)=(0.097,4.807)$.
\end{corollary}
Further admissible tuples are given in table \ref{tab:riemann-tuples}. The computer-assisted calculations used in this article are certified by the accompanying Python script \texttt{certify\_paper.py}, which is available with the ancillary files of the arXiv version of this article at \url{https://arxiv.org/abs/2510.27444}. The code takes a set of parameters to be defined in Lemma \ref{lem:f-bounds} and tests whether these are admissible. Furthermore, it also computes the related constants $(C_1,C_2,C_3)$ for Theorem \ref{thm:main}. With this method, different tuples $(C_1,C_2,C_3)$ can be calculated to optimize the estimate around specific ranges for $T$. The values given in Theorem \ref{thm:main} and Corollary \ref{cor:riemann} are optimized for asymptotic behavior when $T\rightarrow\infty$.

\section{Proof of Theorem 1}
We consider a Dedekind zeta function $\z(s)$ and want to estimate $N_K(T)$. Let us first only consider $T\in\R$ such that $T$ is not the imaginary part of a $\z$-zero. In the case where $T$ is the ordinate of a $\z$-zero, set
\begin{equation}\label{eq:half-weight-continuation}
    N_K(T)=\lim_{\varepsilon\rightarrow0}\frac{N_K(T+\varepsilon)+N_K(T-\varepsilon)}{2}.
\end{equation}
Now, notice that the non-trivial zeros of $\z$ are precisely the zeros of $\xi_K(s)=s(s-1)d_K^{s/2}\gamma_K(s)\z(s)$ and all zeros lie inside the critical strip with real part between $0$ and $1$. This function is holomorphic and thus we can apply the argument principle \cite[Lemma 7.1]{FischerLieb1994} to compute the number of zeros. We have
\begin{equation}\label{eq:argument-principle}
    N_K(T)=\frac{1}{2\pi i}\oint_C\frac{\xi_K'(s)}{\xi_K(s)}ds,
\end{equation}
where $C$ is any contour that surrounds the rectangle $0\pm iT$, $1\pm iT$ in the counterclockwise direction. Let $d>\frac{1}{2}$, which will be fixed later and define $C$ as the contour that connects the four points $\frac{1}{2}\pm d\pm iT$. Notice that by the functional equation \eqref{eq:functional-equation} and the symmetry of the complex conjugate \eqref{eq:argument-principle} simplifies to
\begin{equation*}
    N_K(T)=\frac{2}{\pi i}\oint_{C'}\frac{\xi_K'}{\xi_K}(s)ds=\frac{2}{\pi i}\int_{\frac{1}{2}+d}^{\frac{1}{2}+d+iT}\frac{\xi_K'}{\xi_K}(s)ds+\frac{2}{\pi i}\int_{\frac{1}{2}+d+iT}^{\frac{1}{2}+iT}\frac{\xi_K'}{\xi_K}(s)ds,
\end{equation*}
where $C'$ is the path that starts at $s=\frac{1}{2}+d$, goes to $\frac{1}2+d+iT$ and then to $\frac{1}{2}+iT$. Because $\frac{1}{2}+d>1$, we may compute the first integral with \eqref{eq:completed-zeta} and the Dirichlet series for $\frac{\z'}{\z}(s)$. To this end we want to expand the integral by a copy where the real part is shifted by $d$ such that the entire integral is contained in the right half plane of absolute convergence. This shift enables us to apply an absolutely convergent sum involving the non-trivial $\z$-zeros to evaluate the integral. This approach is inspired by Turing's method \cite[Lemma 1]{Turing1953}. We obtain
\begin{equation}\label{eq:shifted-contour}
    \begin{split}
        N_K(T)=&\frac{2}\pi\int_{\frac{1}{2}+d+iT}^{\frac{1}{2}+iT}\im\left(\frac{\xi_K'}{\xi_K}(s)\right)ds+\frac{2}{\pi}\im\log\left(\xi_K\left(\frac{1}{2}+d+iT\right)\right)\\
        =&-\frac{2}{\pi}\int_\frac{1}{2}^{\frac{1}{2}+d}\im\left(\frac{\xi_K'}{\xi_K}(\sigma+iT)\right)-\im\left(\frac{\xi_K'}{\xi_K}(\sigma+d+iT)\right)d\sigma\\
        &-\frac{2}{\pi}\int_{\frac{1}{2}+d}^{\frac{1}{2}+2d}\im\left(\frac{\xi_K'}{\xi_K}(\sigma+iT)\right)d\sigma+\frac{2}{\pi}\im\log\left(\xi_K\left(\frac{1}{2}+d+iT\right)\right)\\
        =&-\frac{2}{\pi}\int_{\frac{1}{2}}^{\frac{1}{2}+d}\im\left(\frac{\xi_K'}{\xi_K}(\sigma+iT)\right)-\im\left(\frac{\xi_K'}{\xi_K}(\sigma+d+iT)\right)d\sigma+E_1(\xi_K)(T),
    \end{split}
\end{equation}
where $E_1(f)$ is an operator given by
\begin{equation*}
    E_1(f)(t)=\frac{2}{\pi}\left(2\im\log\left( f\left(\frac{1}{2}+d+it\right)\right)-\im\log\left(f\left(\frac{1}{2}+2d+it\right)\right)\right),
\end{equation*}
for a fixed constant $d>\frac{1}{2}$. Here, $\im\log(f)$ is interpreted as the continuous variation of the argument of the function $f$ starting from the real axis. We use the notation $\im\log(f)$ without brackets around the logarithm to signal this difference. Next, we need the following representation for the Dedekind zeta-function to express the remaining integral in \eqref{eq:shifted-contour} via a sum involving the non-trivial $\z$-zeros. Because $\xi_K$ is an entire function of order $1$, we can apply Theorem 5.52 from \cite[page 150]{IwaniecKowalski2004} and obtain that
\begin{equation*}
        \xi_K(s)=e^{Q(s)}\prod_{\rho}\left(1-\frac{s}{\rho}\right)e^\frac{s}{\rho},
\end{equation*}
where $Q(s)$ is a polynomial of degree at most $1$. Hence,
\begin{equation}\label{eq:hadamard-log-derivative}
    \begin{split}
        \frac{\xi_K'}{\xi_K}(s)=Q'(s)+\sum_\rho\left(\frac{1}{\rho}+\frac{1}{s-\rho}\right)
    \end{split}
\end{equation}
where the last sum is absolutely convergent, because there are $O_K(T\log(T))$ terms in the sum in equation \eqref{eq:hadamard-log-derivative} with $|\im(\rho)|=|\omega|\leq T$ and
\begin{equation*}
    \left|\frac{1}{\rho}+\frac{1}{s-\rho}\right|=\left|\frac{s}{\rho(s-\rho)}\right|=\frac{|s|}{|\rho||s-\rho|}=O_s\left(|\rho|^{-2}\right)=O_s(\omega^{-2})
\end{equation*}
for large $\omega=\im(\rho)$. Notice that because $Q(s)$ is a first degree polynomial, $Q'(s)$ is constant. Furthermore, $\frac{\xi_K'}{\xi_K}(s)$ commutes with complex conjugation, as does the sum in equation \eqref{eq:hadamard-log-derivative} because the $\z$-zeros $\rho$ occur in complex conjugate pairs. Therefore, $Q'(s)=Q'(\bar{s})$ and $\im(Q'(s))=0$. Furthermore, we can compute the real part of $Q'(s)$ by exploiting the symmetry of the $\xi_K$-zeros as well as the functional equation \eqref{eq:functional-equation}. Because $Q'(s)$ is constant equations \eqref{eq:hadamard-log-derivative} and the functional equation imply
\begin{equation}\label{eq:Q-prime}
    Q'(s)=\frac{\xi_K'}{\xi_K}(0)=-\frac{\xi_K'}{\xi_K}(1)=-Q'(s)-\sum_{\rho}\left(\frac{1}{\rho}+\frac{1}{1-\rho}\right)=-Q'(s)-2\sum_{\rho}\re\left(\frac{1}{\rho}\right).
\end{equation}
From this we obtain that $Q'(s)=-\sum_\rho\re\left(\frac{1}{\rho}\right)$. We conclude that
\begin{equation*}
    \im\left(\frac{\xi_K'}{\xi_K}(s)\right)=\sum_\rho\frac{\omega-t}{(\sigma-\beta)^2+(t-\omega)^2}-\frac{\omega}{\beta^2+\omega^2},
\end{equation*}
where we again use the notation that $\rho=\beta+i\omega$. With this tool at our disposal, we will find that we can bound the integral in \eqref{eq:shifted-contour} using operators similar to $E_1$ applied to $\xi_K$ and evaluated at points with real parts greater than $1$. Eventually, we will evaluate all of these operators together to take advantage of potential cancellations between these terms. Thus, we obtain
\begin{equation}\label{eq:zero-sum-integral}
    \begin{split}
        N_K(T)=&\frac{2}{\pi}\int_\frac{1}{2}^{\frac{1}{2}+d}\sum_\rho\frac{T-\omega}{(\sigma-\beta)^2+(T-\omega)^2}-\frac{T-\omega}{(\sigma+d-\beta)^2+(T-\omega)^2}d\sigma\\
        &+E_1(\xi_K)(T,d)\\
        =&\frac{2}{\pi}\sum_\rho g(T,\rho)+E_1(\xi_K)(T),
    \end{split}
\end{equation}
where
\begin{equation}\label{eq:g-definition}
    \begin{split}
        g(t,\rho)=&\int_\frac{1}{2}^{\frac{1}{2}+d}\frac{t-\omega}{(\sigma-\beta)^2+(t-\omega)^2}-\frac{t-\omega}{(\sigma+d-\beta)^2+(t-\omega)^2}d\sigma\\
        =&-\arctan\left(\frac{\frac{1}{2}+2d-\beta}{t-\omega}\right)+2\arctan\left(\frac{\frac{1}{2}+d-\beta}{t-\omega}\right)-\arctan\left(\frac{\frac{1}{2}-\beta}{t-\omega}\right),
    \end{split}
\end{equation}
for a given fixed constant $d>\frac{1}{2}$ and again we used the notation $\rho=\beta+i\omega$. Here we use the fact that we consider the finite integral of an absolutely convergent sum to swap the order of summation and integration. Now, due to the functional equation for $\z(s)$ and the fact that $\z(s)$ commutes with complex conjugation, we obtain $\z(1-\bar{\rho})=0$ for every $\z$-zero $\rho$. Furthermore, we observe that the sum in equation \eqref{eq:zero-sum-integral} is again absolutely convergent. Therefore, we may reorder terms and thus group terms associated with $\rho$ and $1-\bar{\rho}$ together. Doing so allows us to cancel the last term from equation \eqref{eq:g-definition}, which yields
\begin{equation}\label{eq:NK-f-sum}
    N_K(T)=\frac{1}{\pi}\sum_\rho f(b,T-\omega)+E_1(\xi_K)(T),
\end{equation}
where
\begin{equation*}
    \begin{split}
    f(b,t)=&2\arctan\left(\frac{b+d}{t}\right)+2\arctan\left(\frac{-b+d}{t}\right)\\
    &-\arctan\left(\frac{b+2d}{t}\right)-\arctan\left(\frac{-b+2d}{t}\right),
    \end{split}
\end{equation*}
for a given fixed constant $d>\frac{1}{2}$. Here $b:=\frac{1}{2}-\beta$. Notice that this defines a harmonic function in $b$ and $t$ for $t\neq0$. Now, the idea is to limit the contribution of each term in equation \eqref{eq:NK-f-sum} by a linear combination of poles of different degrees located at $\frac{1}{2}+d+iT-\rho$. We may then evaluate the sum of these poles due to their connection to the logarithmic derivatives of $\xi_K(\frac{1}{2}+d+iT)$. This idea was inspired by a very similar method used by Turing to estimate $S_1(T)$ \cite{Turing1953}. The following lemma presents us with the required bounds for $f(b,t)$.
\begin{lemma}\label{lem:f-bounds}
    Let $|b|\leq\frac{1}{2}$ and $t\neq0$. Then 
    \begin{equation*}
        \begin{split}
            f(b,t)\leq&\frac{\pi}{4}\bigg[da_1\left(\frac{d+b}{(d+b)^2+t^2}+\frac{d-b}{(d-b)^2+t^2}\right)\\
            &+d^2a_2\left(\frac{(d+b)^2-t^2}{((d+b)^2+t^2)^2}+\frac{(d-b)^2-t^2}{((d-b)^2+t^2)^2}\right)\\
            &+a_3\left(\frac{2(d+b)t}{((d+b)^2+t^2)^2}+\frac{2(d-b)t}{((d-b)^2+t^2)^2}\right)\bigg]\ \text{ and}\\
            f(b,t)\geq&-\frac{\pi}{4}\bigg[da_1\left(\frac{d+b}{(d+b)^2+t^2}+\frac{d-b}{(d-b)^2+t^2}\right)\\
            &+d^2a_2\left(\frac{(d+b)^2-t^2}{((d+b)^2+t^2)^2}+\frac{(d-b)^2-t^2}{((d-b)^2+t^2)^2}\right)\\
            &-a_3\left(\frac{2(d+b)t}{((d+b)^2+t^2)^2}+\frac{2(d-b)t}{((d-b)^2+t^2)^2}\right)\bigg]
        \end{split}
    \end{equation*}
    for $d=0.722$, $a_1=1.07$, $a_2=0.93$ and $a_3=0.365$.
\end{lemma}
\begin{proof}
    We prove the upper bound first and conclude the lower bound by symmetry. Let
    \begin{equation*}
        \begin{split}
            H(b,t):=f(b,t)-&\frac{\pi}{4}\bigg[da_1\left(\frac{d+b}{(d+b)^2+t^2}+\frac{(d-b)}{(d-b)^2+t^2}\right)\\
            +&d^2a_2\left(\frac{(d+b)^2-t^2}{((d+b)^2+t^2)^2}+\frac{(d-b)^2-t^2}{((d-b)^2+t^2)^2}\right)\\
            +&a_3\left(\frac{2(d+b)t}{((d+b)^2+t^2)^2}+\frac{2(d-b)t}{((d-b)^2+t^2)^2}\right)\bigg].
        \end{split}
    \end{equation*} 
    Notice that the functions in the brackets may be expressed as the real and imaginary parts of $\frac{1}{d+b+it}$ and $\frac{1}{(d+b+it)^2}$, therefore, they are harmonic functions in $b$ and $t$ for $d+b+it\neq0$ and a fixed $d$. For fixed $a_1$, $a_2$, and $a_3$, $H(b,t)$ is then a harmonic function for $t\neq0$ because $f(b,t)$ is a harmonic function for $t\neq0$. Notice that it suffices to show that
    \begin{equation}
        \label{eq:H-nonpositive}
        H(b,t)\leq0
    \end{equation}
    for all $(b,t)\in\left[-\frac{1}{2},\frac{1}{2}\right]\times(\R\setminus{0})$ in order to prove the lemma. First, notice that we obtain from the Taylor series of $\arctan(x)$ that
    \begin{equation}\label{eq:H-asymptotic}
        H(b,t)=\frac{d^2\pi}{2t^2}(a_2-a_1)+O(t^{-3})
    \end{equation}
    uniformly in $b$ for $b\in[-0.5,0.5]$. This verifies the inequality \eqref{eq:H-nonpositive} for large $|t|$, say $|t|\geq T_0,$ because $a_1>a_2$. Because $H(b,t)$ is harmonic for $t\neq0$, we may apply the maximum principle to verify \eqref{eq:H-nonpositive} for $(b,t)\in\left[-\frac{1}{2},\frac{1}{2}\right]\times([-T_0,0)\cup(0,T_0])$. We conclude that it suffices to show that $H(b,t)\leq0$ on the boundary of this area, in order to establish the desired result. First, consider the area with $t>0$. By our choice of $T_0$ we already know that $H(b,T_0)\leq0$ for all $b\in\left[-\frac{1}{2},\frac{1}{2}\right]$. Because $d>\frac{1}{2}$ and $|b|\leq\frac{1}{2}$, we have that $d\pm b>0$ for all $b\in\left[-\frac{1}{2},\frac{1}{2}\right]$. Thus,
    \begin{equation*}
        \begin{split}
            \lim_{t\rightarrow0^+}f(b,t)=\lim_{t\rightarrow0^+}\bigg[&2\arctan\left(\frac{b+d}{t}\right)+2\arctan\left(\frac{-b+d}{t}\right)\\
            &-\arctan\left(\frac{b+2d}{t}\right)-\arctan\left(\frac{-b+2d}{t}\right)\bigg]=\pi.
        \end{split}
    \end{equation*}
    And therefore also
    \begin{equation*}
        \begin{split}
            \lim_{t\rightarrow0^+}H(b,t)=&\pi-\frac{\pi}{4}\left[da_1\left(\frac{1}{d+b}+\frac{1}{d-b}\right)+d^2a_2\left(\frac{1}{(d+b)^2}+\frac{1}{(d-b)^2}\right)\right]\\
            \leq&\pi-\frac{\pi}{4}\left[da_1\left(\frac{1}{d}+\frac{1}{d}\right)+d^2a_2\left(\frac{1}{d^2}+\frac{1}{d^2}\right)\right]\\
            \leq&\pi-\frac{\pi}{4}\left[2a_1+2a_2\right]=0.
        \end{split}
    \end{equation*}
    Next, observe that $H(b,t)$ is symmetric about $b=0$. Therefore, it suffices to show that $H\left(\frac{1}{2},t\right)\leq0$ for all $t>0$ to verify that $H(b,t)\leq0$ on the boundary of $\left[-\frac{1}{2},\frac{1}{2}\right]\times(0,T_0]$. Let
    \begin{equation}\label{eq:def-h}
        \begin{split}
            h(t):=H\left(\frac{1}{2},t\right)=&2\arctan\left(\frac{d+\frac{1}{2}}{t}\right)+2\arctan\left(\frac{d-\frac{1}{2}}{t}\right)\\
            &-\arctan\left(\frac{2d+\frac{1}{2}}{t}\right)-\arctan\left(\frac{2d-\frac{1}{2}}{t}\right)\\
            &-\frac{\pi}{4}\left[da_1\left(\frac{d+\frac{1}{2}}{\left(d+\frac{1}{2}\right)^2+t^2}+\frac{d-\frac{1}{2}}{\left(d-\frac{1}{2}\right)^2+t^2}\right)\right.\\
            &+d^2a_2\left(\frac{\left(d+\frac{1}{2}\right)^2-t^2}{\left(\left(d+\frac{1}{2}\right)^2+t^2\right)^2}+\frac{\left(d-\frac{1}{2}\right)^2-t^2}{\left(\left(d-\frac{1}{2}\right)^2+t^2\right)^2}\right)\\
            &\left.+a_3\left(\frac{2\left(d+\frac{1}{2}\right)t}{\left(\left(d+\frac{1}{2}\right)^2+t^2\right)^2}+\frac{2\left(d-\frac{1}{2}\right)t}{\left(\left(d-\frac{1}{2}\right)^2+t^2\right)^2}\right)\right].
        \end{split}
    \end{equation}
    Next, we show that $h(t)<0$ for all $t\in\R\setminus\{0\}$. Since $T_0$ is not explicit, we do not restrict the verification to $|t|\leq T_0$, but prove the inequality for all $t\neq0$ by a computer-assisted argument, separately for each parameter tuple $(d,a_1,a_2,a_3)$ used in this paper. The code uses the ball arithmetic of Arb~\cite{Johansson2017} and confirms for the given tuples that $h(t)<0$ for all $t\in\R\setminus\{0\}$. This is done by covering the tail of the lines $(0,\infty)$ and $(-\infty,0)$ by a Taylor argument  and then covering the remaining finite reach by dyadic boxes in such a way that we can show with Taylor's theorem that $h(t)$ is negative in every such box. To achieve this the code subdivides the dyadic boxes further and further until a negative covering is achieved. Hence, a successful run is a proof of the inequality, assuming the correctness of the software. To show that $h(t)<0$ for $t$ with large absolute value we consider the substitution $F(u)=h\left(\frac{1}{u}\right)$. $F(u)$ extends to a real, analytic function near $u=0$, and due to equation \eqref{eq:H-asymptotic} we know that $F(0)=F'(0)=0$ and $F''(0)=-2r$ where $r=\frac{\pi}{2}d^2(a_1-a_2)>0$. For each tuple $(d,a_1,a_2,a_3)$ we verify with interval arithmetic that $-3r<F''(u)<-r$ for all $u\in\left[-2^{-k},2^{-k}\right]$ where $k=7$. We can then apply Taylor's theorem to obtain that $F(u)=\frac{1}{2}F''(v)u^2$ for some $v$ between $0$ and $u$. Therefore, we obtain that 
    \begin{equation}\label{eq:h-tail-negativity}
        -\frac{3r}{2t^2}\leq h(t)\leq-\frac{r}{2t^2}<0 \text{ for all }|t|\geq2^k.
    \end{equation}
    To cover the remaining finite region we define
    \begin{equation}\label{eq:h+-definition}
        \begin{split}
            h_{\pm}(t)=&\arctan\left(\frac{t}{2d+\frac{1}{2}}\right)+\arctan\left(\frac{t}{2d-\frac{1}{2}}\right)-2\arctan\left(\frac{t}{d+\frac{1}{2}}\right)-2\arctan\left(\frac{t}{d-\frac{1}{2}}\right)\\
            &\pm\pi-\frac{\pi}{4}R(t),
        \end{split}
    \end{equation}
    where $R(t)$ is the expression in square brackets in equation \eqref{eq:def-h}.
    These are both real analytic functions on all of $\R$ and because $\arctan(x)=\text{sgn}(x)\frac{\pi}{2}-\arctan\left(\frac{1}{x}\right)$ we see that $h=h_+$ for $t>0$ and $h=h_-$ for $t<0$. Now it suffices to show $h_+<0$ on $[0,2^k]$ and $h_-<0$ on $[-2^k,0]$ to establish $h(t)<0$ for all $t\neq0$. Let $g\in\{h_+,h_-\}$ and let $[m-w,m+w]$ be a box in $[0,2^k]$ (for $g=h_+$) or in $[-2^k,0]$ (for $g=h_-$). By Taylor's theorem, for all $x$ in the box,
    \begin{equation*}
        g(x)\leq g(m)+|g'(m)|\,w+\frac{1}{2}\max\left\{0,\sup g''\right\}\,w^2,
    \end{equation*}
    where the supremum is taken over the box. We compute $g(m)$, $g'(m)$ and an enclosure of $g''$ over the box by automatic differentiation in ball arithmetic. We accept the box if the right-hand side is negative, and we bisect it otherwise. All box endpoints are exact binary numbers, so the accepted boxes cover $[0,2^k]$ and $[-2^k,0]$ without gaps. For each tuple the procedure stops after finitely many steps with all boxes accepted, which proves $h_+<0$ on $[0,2^k]$ and $h_-<0$ on $[-2^k,0]$. Together with equation \eqref{eq:h-tail-negativity} we deduce that $h(t)<0$ for all $t\neq0$, and in particular $H\left(\frac{1}{2},t\right)=H\left(-\frac{1}{2},t\right)=h(t)\leq0$ for all $t\in\R\setminus\{0\}$. Thus, we obtain by the maximum principle that $H(b,t)\leq0$ for $b\in\left[-\frac{1}{2},\frac{1}{2}\right]$ and $0<t\leq T_0$. For the area with $t<0$ we obtain $H(b,-T_0)\leq 0$ by our choice of $T_0$. Furthermore, notice that
    \begin{equation*}
        \begin{split}
            \lim_{t\rightarrow0^-}H(b,t)=&-\pi-\frac{\pi}{4}\left[da_1\left(\frac{1}{d+b}+\frac{1}{d-b}\right)+d^2a_2\left(\frac{1}{(d+b)^2}+\frac{1}{(d-b)^2}\right)\right]\\
            \leq&-\pi-\frac{\pi}{4}\left[da_1\left(\frac{1}{d}+\frac{1}{d}\right)+d^2a_2\left(\frac{1}{d^2}+\frac{1}{d^2}\right)\right]\\
            \leq&-\pi-\frac{\pi}{4}\left[2a_1+2a_2\right]<0.
        \end{split}
    \end{equation*}
    And by the maximum principle, we conclude that $H(b,t)\leq0$ for $b\in\left[-\frac{1}{2},\frac{1}{2}\right]$ and $-T_0\leq t<0$. This proves the upper bound for $f(b,t)$. Now note that $f(b,t)$ and $\frac{2(d+b)t}{((d+b)^2+t^2)^2}$ are odd functions in $t$, while $\frac{d+b}{(d+b)^2+t^2}$ and $\frac{(d+b)^2-t^2}{((d+b)^2+t^2)^2}$ are even functions in $t$. Furthermore, we have established the upper bound for all $|b|\leq\frac{1}{2}$ and $t\neq0$. Hence, the lower bound follows from the upper bound by substituting $t\mapsto -t$.
\end{proof}
\begin{remark}
    We will see that the quality of the overall estimate of Theorem \ref{thm:main} depends linearly on the size of the product $da_1$. All we did in Lemma \ref{lem:f-bounds} was to show that these particular values for $d$, $a_1$, $a_2$, and $a_3$ are admissible values for the optimization problem, that is, to find values for $d$, $a_1$, $a_2$, and $a_3$ such that $da_1$ is minimal under the condition that $H(b,t)\leq0$ on the boundary of $\left[-\frac{1}{2},\frac{1}{2}\right]\times\R$. The condition requires $a_1>a_2$ and $a_1+a_2\geq2$ to control the behavior of $H(b,t)$ for $|t|\rightarrow\infty$ and $|t|\rightarrow0$. We see that the terms associated with $a_1$ and $a_2$ are even functions in $t$ and that $f(b,t)$ is an odd function in $t$ that is negative for negative $t$. Hence, we use the $a_3$ term to distribute the weight of the estimate more evenly across the positive and negative values for $t$. Without this wild card, inequality $H(b,t)\leq0$ is trivially true for all $t<0$. The specific values of $d$, $a_1$, $a_2$, and $a_3$ for this lemma were found by strategic testing.
\end{remark}
\begin{remark}
    Further admissible choices for $(d,a_1,a_2,a_3)$ are given in table \ref{tab:parameters} below. We index ($j$) the tuples to allow for easier cross-referencing throughout the paper.
    \begin{table}[h!]
    \centering
    \begin{tabular}{|c|c|c|c|c|}
    \hline
    $j$ & $d$ & $a_1$ & $a_2$ & $a_3$ \\ [0.5ex] 
    \hline
    1 & $0.778$ & $1.028$ & $0.972$ & $0.553$ \\
    2 & $0.922$ & $1.041$ & $0.959$ & $0.708$ \\
    3 & $1.065$ & $1.051$ & $0.949$ & $0.882$ \\
    4 & $1.203$ & $1.064$ & $0.936$ & $1.034$ \\
    5 & $1.332$ & $1.081$ & $0.919$ & $1.136$ \\
    6 & $1.461$ & $1.095$ & $0.905$ & $1.247$ \\
    [1ex]  
    \hline
    \end{tabular}
    \caption{Further admissible choices for $(d,a_1,a_2,a_3)$}
    \label{tab:parameters}
    \end{table}\\
    These values were chosen to minimize 
    \begin{equation*}
        \begin{split}
            &\frac{1}{\pi}\left(\log(d)+2\log\left(2d-\frac{1}{2}\right)-4\log\left(d-\frac{1}{2}\right)\right)\\
            +&\frac{da_1}{2}\left(\frac{1}{d-\frac{1}{2}}-\frac{1}{d}\right)+\frac{\sqrt{(d^2a_2)^2+a_3^2}}{2}\left(\frac{1}{\left(d-\frac{1}{2}\right)^2}-\frac{1}{2d^2}\right)
        \end{split}
    \end{equation*}
    for $da_1\leq0.64+0.16j$ by strategic testing. The choice was to minimize this quantity because the dominant contribution to $C_2$ term in Theorem \ref{thm:main} is roughly proportional to this number, while $C_1$ is linearly dependent on the product $da_1$.
\end{remark}
Now, with these bounds we may estimate the sum in \eqref{eq:NK-f-sum}. Recall that we use $\rho=\beta+i\omega$ and $b=\frac{1}{2}-\beta$. We obtain the upper bound given by
\begin{equation}\label{eq:zero-sum-upper}
    \begin{split}
        \frac{1}{\pi}\sum_\rho f(b,T-\omega)\leq&\frac{da_1}{4}\re\left(\sum_\rho\frac{1}{\frac{1}{2}+d+iT-\rho}+\frac{1}{\frac{1}{2}+d+iT-(1-\bar{\rho})}\right)\\
        &+\frac{d^2a_2}{4}\re\left(\sum_\rho\frac{1}{(\frac{1}{2}+d+iT-\rho)^2}+\frac{1}{(\frac{1}{2}+d+iT-(1-\bar{\rho}))^2}\right)\\
        &-\frac{a_3}{4}\im\left(\sum_\rho\frac{1}{(\frac{1}{2}+d+iT-\rho)^2}+\frac{1}{(\frac{1}{2}+d+iT-(1-\bar{\rho}))^2}\right)\\
        =&E_2(\xi_K)(T)+E_3(\xi_K)(T)
    \end{split}
\end{equation}
and analogously the lower bound given by
\begin{equation}\label{eq:zero-sum-lower}
    \begin{split}
        \frac{1}{\pi}\sum_\rho f(b,T-\omega)\geq&-\frac{da_1}{4}\re\left(\sum_\rho\frac{1}{\frac{1}{2}+d+iT-\rho}+\frac{1}{\frac{1}{2}+d+iT-(1-\bar{\rho})}\right)\\
        &-\frac{d^2a_2}{4}\re\left(\sum_\rho\frac{1}{(\frac{1}{2}+d+iT-\rho)^2}+\frac{1}{(\frac{1}{2}+d+iT-(1-\bar{\rho}))^2}\right)\\
        &-\frac{a_3}{4}\im\left(\sum_\rho\frac{1}{(\frac{1}{2}+d+iT-\rho)^2}+\frac{1}{(\frac{1}{2}+d+iT-(1-\bar{\rho}))^2}\right)\\
        =&-E_2(\xi_K)(T)+E_3(\xi_K)(T).
    \end{split}
\end{equation}
Here, $E_2(f)(t)$ and $E_3(f)(t)$ are operators given by
\begin{equation*}
    \begin{split}
        E_2(f)(t):=&\left[\frac{da_1}{2}\re\left(\frac{f'}{f}\right)-\frac{d^2a_2}{2}\re\left(\left(\frac{f'}{f}\right)'\right)\right]\left(\frac{1}{2}+d+it\right)\text{ and}\\
        E_3(f)(t):=&\frac{a_3}{2}\im\left(\left(\frac{f'}{f}\right)'\right)\left(\frac{1}{2}+d+it\right).
    \end{split}
\end{equation*}
Notice that $E_1,E_2$ and $E_3$ satisfy $E_i(f\cdot g)(t)=E_i(f)(t)+E_i(g)(t)$. Equations \eqref{eq:zero-sum-upper} and \eqref{eq:zero-sum-lower} follow due to this fact and equations \eqref{eq:hadamard-log-derivative} and \eqref{eq:Q-prime}. We may now decompose $\xi_K(s)$ according to \eqref{eq:completed-zeta}. We thus obtain for $E_u:=E_1+E_2+E_3$ and $E_l:=E_1-E_2+E_3$, that
\begin{equation}\label{eq:xi-decomposition}
    \begin{split}
       E_u(\xi_K)(T)=&E_u(s(s-1))(T)+E_u\left(d_K^{s/2}\right)(T)+E_u(\gamma_K)(T)+E_u(\z)(T)\ \text{ and}\\
       E_l(\xi_K)(T)=&E_l(s(s-1))(T)+E_l\left(d_K^{s/2}\right)(T)+E_l(\gamma_K)(T)+E_l(\z)(T).
    \end{split}
\end{equation}
By equation \eqref{eq:zero-sum-integral} we obtain that
\begin{equation*}
    E_l(\xi_K)(T)\leq N_K(T)\leq E_u(\xi_K)(T).
\end{equation*}
Next, we will use equation \eqref{eq:xi-decomposition} to estimate each term individually.
\begin{lemma}\label{lem:polynomial-factor}
    For $d,a_1,a_2,a_3$ as in Lemma \ref{lem:f-bounds} and $T\geq1$ we obtain
    \begin{equation*}
        \begin{split}
            E_u(s(s-1))(T)\leq&2.00208\ \text{ and}\\
            E_l(s(s-1))(T)\geq&1.83865.
        \end{split}
    \end{equation*}
\end{lemma}
\begin{proof}
    Plugging in the equation for $E_u(s(s-1))(T)$ and $E_l(s(s-1))(T)$ yield
    \begin{equation*}
        \begin{split}
            E_u(s(s-1))(T)=&\frac{4}{\pi}\left(\arctan\left(\frac{T}{d+\frac{1}{2}}\right)+\arctan\left(\frac{T}{d-\frac{1}{2}}\right)\right)\\
            &-\frac{2}{\pi}\left(\arctan\left(\frac{T}{2d+\frac{1}{2}}\right)+\arctan\left(\frac{T}{2d-\frac{1}{2}}\right)\right)\\
            &+\frac{da_1}{2}\left(\frac{d+\frac{1}{2}}{(d+\frac{1}{2})^2+T^2}+\frac{d-\frac{1}{2}}{(d-\frac{1}{2})^2+T^2}\right)\\
            &+\frac{d^2a_2}{2}\left(\frac{(d+\frac{1}{2})^2-T^2}{((d+\frac{1}{2})^2+T^2)^2}+\frac{(d-\frac{1}{2})^2-T^2}{((d-\frac{1}{2})^2+T^2)^2}\right)\\
            &+\frac{a_3}{2}\left(\frac{2(d+\frac{1}{2})T}{((d+\frac{1}{2})^2+T^2)^2}+\frac{2(d-\frac{1}{2})T}{((d-\frac{1}{2})^2+T^2)^2}\right)\\
            =&2-\frac{2}{\pi}h(T),\qquad\text{ and}\\
            E_l(s(s-1))(T)=&2+\frac{2}{\pi}h(-T).
        \end{split}
    \end{equation*}
    This means that we need upper bounds $S_+$ and $S_-$ such that $-h(t)\leq S_+$ for all $t\in[1,\infty)$ and $-h(t)\leq S_-$ for all $t\in(-\infty,-1]$. We obtain such bounds with a similar method like in the proof of Lemma \ref{lem:f-bounds}. Let $r=\frac{\pi}{2}d^2(a_1-a_2)$, for $|t|\geq2^k$ the tail estimate \eqref{eq:h-tail-negativity} gives
    \begin{equation*}
        0<-h(t)\leq\frac{3r}{2t^2}\leq\frac{3r}{2^{2k+1}},
    \end{equation*}
    which is at most $1.05\cdot10^{-5}$ for the main tuple and $k=7$. For $1\leq t\leq2^k$ and for $-2^k\leq t\leq-1$ we use the box bound of that proof with $g=-h_+$ and $g=-h_-$, respectively. For every $C^2$ function $g$ and every $x\in[m-w,m+w]$,
    \begin{equation*}
        g(x)\leq g(m)+|g'(m)|w+\tfrac12\max\{0,\sup g''\}w^2=:U,
    \end{equation*}
    where the supremum is taken over the box and all quantities are enclosed in ball arithmetic. Instead of accepting a box as soon as $U<0$, we now repeatedly bisect the box with the largest $U$, until this $U$ exceeds the largest computed value of $g$ at a midpoint or an endpoint by at most $10^{-13}$. The largest $U$ over all boxes is then a rigorous upper bound for $\sup g$ on the interval. Here, the stopping rule affects only how sharp this bound is, not whether it is valid. For the main tuple this gives,
    \begin{equation*}
        \sup_{1\leq t\leq2^k}\left(-h(t)\right)\leq0.0032647,\qquad
        \sup_{-2^k\leq t\leq-1}\left(-h(t)\right)\leq0.2534472.
    \end{equation*}
    Both bounds are larger than the tail bound, so we may take $S_+=0.0032647$ and $S_-=0.2534472$, and obtain for $T\geq1$
    \begin{equation*}
        E_u(s(s-1))(T)\leq2+\frac{2}{\pi}\cdot0.0032647<2.00208,
        \quad
        E_l(s(s-1))(T)\geq2-\frac{2}{\pi}\cdot0.2534472>1.83865.
    \end{equation*}
\end{proof}
Next, we estimate the discriminant term.
\begin{lemma}\label{lem:discriminant-factor}
    For $d,a_1,a_2,a_3$ as in Lemma \ref{lem:f-bounds} and $T\geq1$ we obtain
    \begin{equation*}
        \begin{split}
            E_u\left(d_K^{s/2}\right)(T)=&\left(\frac{T}{\pi}+\frac{da_1}{4}\right)\log(d_K)\ \text{ and}\\
            E_l\left(d_K^{s/2}\right)(T)=&\left(\frac{T}{\pi}-\frac{da_1}{4}\right)\log(d_K),
        \end{split}
    \end{equation*}
    where $d_K$ denotes the absolute discriminant of the number field $K$.
\end{lemma}
\begin{proof}
    Plugging in the equations for $E_u\left(d_K^{s/2}\right)(T)$ and $E_l\left(d_K^{s/2}\right)(T)$ yields the result. We have
    \begin{equation*}
        \begin{split}
            E_u\left(d_K^{s/2}\right)(T)=&\frac{4}{\pi}\im\log\left(d_K^{\frac{\frac{1}{2}+d+iT}{2}}\right)-\frac{2}{\pi}\im\log\left(d_K^{\frac{\frac{1}{2}+2d+iT}{2}}\right)+\frac{da_1}{2}\re\left(\frac{\log(d_K)}{2}\right)\\
            =&\left(\frac{T}{\pi}+\frac{da_1}{4}\right)\log(d_K)\ \text{ and}\\
            E_l\left(d_K^{s/2}\right)(T)=&\frac{4}{\pi}\im\log\left(d_K^{\frac{\frac{1}{2}+d+iT}2}\right)-\frac{2}{\pi}\im\log\left(d_K^{\frac{\frac{1}{2}+2d+iT}2}\right)-\frac{da_1}{2}\re\left(\frac{\log(d_K)}{2}\right)\\
            =&\left(\frac{T}{\pi}-\frac{da_1}{4}\right)\log(d_K).
        \end{split}
    \end{equation*}
\end{proof}
Before we estimate the contribution of the gamma term $\gamma_K(s)$, we require some useful estimates for the logarithmic gamma function and its derivatives. The logarithmic gamma function is a special function, in the sense that it is not any particular branch of the logarithm applied to the gamma function, but instead it is defined as the integral $\int_1^s\frac{\Gamma'}{\Gamma}(z)dz$ with a branch cut along the negative real axis. This distinction is important to us because we specifically defined $\im\log(f)$ as the continuous variation of the argument of the function $f$ starting from the real axis. So, in particular $\im\log(\Gamma(s))$ will be the imaginary part of the logarithmic gamma function.
\begin{proposition}\label{prop:gamma-estimates}
    Let $s\in\C$ with $\re(s)>0$, $s=\sigma+it$. We then have
    \begin{equation*}
        \begin{split}
            \im\log(\Gamma(s))=&\frac{t}{2}\log\left(\sigma^2+t^2\right)+\left(\sigma-\frac{1}{2}\right)\arctan\left(\frac{t}{\sigma}\right)-t-\frac{t}{12(\sigma^2+t^2)}+R_1(s),\\
            \re\left(\psi(s)\right)=&\frac{1}{2}\log\left(\sigma^2+t^2\right)-\frac{\sigma}{2(\sigma^2+t^2)}-\frac{\sigma^2-t^2}{12(\sigma^2+t^2)^2}+R_2(s),\\
            \re(\psi_1(s))=&\frac{\sigma}{\sigma^2+t^2}+\frac{\sigma^2-t^2}{2(\sigma^2+t^2)^2}+R_3(s),\text{ and}\\
            \im(\psi_1(s))=&-\frac{t}{\sigma^2+t^2}-\frac{\sigma t}{(\sigma^2+t^2)^2}+R_4(s).
        \end{split}
    \end{equation*}
    Here,
    \begin{align*}
        \left|R_1(s)\right|&\leq\frac{1}{360|s|^3}+\frac{1}{1011\sigma|s|^3},\\
        |R_2(s)|&\leq\frac{1}{120\sigma|s|^3},\text{ and}\\
        |R_3(s)|,|R_4(s)|&\leq\frac{1}{6|s|^3}+\frac{1}{22\sigma|s|^3}.
    \end{align*}
\end{proposition}
Before we prove this, we need a short lemma with some estimates for some auxiliary functions that we will need to prove proposition \ref{prop:gamma-estimates}.
\begin{lemma}\label{lem:binet-remainders}
    Let $g(u)=\frac12-\frac1u+\frac1{e^u-1}$, $f(u)=g(u)/u$ and $h(u)=ug(u)$. Then the following estimates hold for all $u>0$
    \begin{align*}
        |g'''(u)|\leq&\frac1{120},\\
        |f'''(u)|\leq&\frac1{1011},\text{ and}\\
        |h'''(u)|\leq&\frac1{22}.
    \end{align*}
\end{lemma}
\begin{proof}
    Define the sequence $(a_k)_{k\in\Z_{\geq1}}$ by $a_k=2\pi k$ for all $k\geq1$. The partial fraction expansion of the $\coth(\cdot)$ function gives
    \begin{align*}
        g(u)=&\sum_{k\geq1}\frac{2u}{u^2+a_k^2},\\
        f(u)=&\sum_{k\geq1}\frac{2}{u^2+a_k^2},\text{ and}\\
        h(u)=&\sum_{k\geq1}2\left(1-\frac{a_k^2}{u^2+a_k^2}\right).
    \end{align*}
    These series converge absolutely for all $u$ in a compact interval $u\in[0,R]$ with some $0<R<\infty$ by comparison with $C_R\sum_{k\geq1}a_k^{-2}$. By the Weierstrass M-test we obtain that we may differentiate termwise. Notice further that the terms of the derivatives will also have decay of order $O\left(a_k^{-2}\right)$ and thus we can write the third derivatives by termwise differentiation. Let $\chi(t)=\frac t{1+t^2}$ and $\varphi(t)=\frac1{1+t^2}$. Then
    $\chi^{(3)}(t)=-\frac{6(t^4-6t^2+1)}{(1+t^2)^4}$ and $\varphi^{(3)}(t)=\frac{24t(1-t^2)}{(1+t^2)^4}$, and
    \begin{align*}
        g^{(3)}(u)=&\sum_k\frac{2}{a_k^{4}}\chi^{(3)}\!\left(\tfrac u{a_k}\right),\\
        f^{(3)}(u)=&\sum_k\frac{2}{a_k^{5}}\varphi^{(3)}\!\left(\tfrac u{a_k}\right),\text{ and}\\
        h^{(3)}(u)=&-\sum_k\frac{2}{a_k^{3}}\varphi^{(3)}\!\left(\tfrac u{a_k}\right).
    \end{align*}
    Since $\left|t^4-6t^2+1\right|\leq(1+t^2)^2$, we have $\left|\chi^{(3)}\right|\leq6$. Furthermore, $\varphi^{(4)}(t)=\frac{24(1-10t^2+5t^4)}{(1+t^2)^5}$ vanishes at $t^2=1\pm\frac{2}{\sqrt5}$, so
    $\left|\varphi^{(3)}\right|\leq\varphi^{(3)}(t_1)=\frac{48\sqrt{\frac{1}{5}\left(1-\frac{2}{\sqrt{5}}\right)}}{\left(2-\frac{2}{\sqrt{5}}\right)^4}$ with $t_1^2=1-\frac2{\sqrt5}$ (the other critical value is approximately $-0.421$). Therefore,
    \begin{align*}
        \left|g^{(3)}\right|\leq&12\sum a_k^{-4}=\frac{12\zeta(4)}{(2\pi)^4}=\frac{1}{120},\\
        \left|f^{(3)}\right|\leq&\frac{3\zeta(5)\sqrt{\frac{1}{5}\left(1-\frac{2}{\sqrt{5}}\right)}}{\left(2-\frac{2}{\sqrt{5}}\right)^4\pi^5}<\frac1{1011},\text{ and}\\
        \left|h^{(3)}\right|\leq&\frac{12\zeta(3)\sqrt{\frac{1}{5}\left(1-\frac{2}{\sqrt{5}}\right)}}{\left(2-\frac{2}{\sqrt{5}}\right)^4\pi^3}<\frac1{22},
    \end{align*}
    as desired.
\end{proof}
\begin{remark}
    The $f$ and $h$ constants can be slightly optimized but they suffice for our purpose here.
\end{remark}
\begin{proof}[Proof of Proposition \ref{prop:gamma-estimates}.]
    These estimates rely on Binet's first formula for the logarithmic gamma function \cite[Theorem 12.31]{WhittakerWatson1920} which is valid for $\re(s)>0$. With this formula for the logarithmic gamma function and the notion that the imaginary part of the logarithmic gamma function is exactly $\im\log(\Gamma(s))$ as we defined it for equation \eqref{eq:shifted-contour} we obtain 
    \begin{equation}\label{eq:binet-imag}
        \im\log(\Gamma(s))=\frac{t}{2}\log(\sigma^2+t^2)+\left(\sigma-\frac{1}{2}\right)\arctan\left(\frac{t}{\sigma}\right)-t+\im\left(\int_0^\infty f(u)e^{-us}du\right),
    \end{equation}
    where $f(u)=\frac{1}{u}\left(\frac{1}{2}-\frac{1}{u}+\frac{1}{e^{u}-1}\right)$ like in lemma \ref{lem:binet-remainders}. Next, we can observe that $\lim_{u\rightarrow0}f(u)=\frac{1}{12}$, $\lim_{u\rightarrow0}f'(u)=0$, $\lim_{u\rightarrow0}f''(u)=-\frac{1}{360}$, and $|f'''(u)|\leq\frac{1}{1011}$ for $u>0$. Now, take equation \eqref{eq:binet-imag} and apply partial integration. We thus obtain that
    \begin{equation*}
        \begin{split}
            \int_0^\infty\left[\frac{1}{u}\left(\frac{1}{2}-\frac{1}{u}+\frac{1}{e^{u}-1}\right)\right]e^{-us}du=&\int_0^\infty f(u)e^{-us}du\\
            =&\frac{1}{12s}-\frac{1}{360s^3}+\frac{1}{s^3}\int_0^\infty f'''(u)e^{-us}du.
        \end{split}
    \end{equation*}
    Hence,
    \begin{equation*}
        \begin{split}
            \left|R_1(s)\right|=\left|\im\left(\int_0^\infty\left[\frac{1}{u}\left(\frac{1}{2}-\frac{1}{u}+\frac{1}{e^{u}-1}\right)\right]e^{-us}du-\frac{1}{12s}\right)\right|\leq\frac{1}{360|s|^3}+\frac{1}{1011\sigma|s|^3}.
        \end{split}
    \end{equation*}
    Considering the derivative of Binet's first formula for the logarithmic gamma function yields that
    \begin{equation}\label{eq:digamma-real}
        \re(\psi(s))=\frac{1}{2}\log(\sigma^2+t^2)-\frac{\sigma}{2(\sigma^2+t^2)}-\re\left(\int_0^\infty g(u)e^{-us}du\right),
    \end{equation}
    where $g(u)=\frac{1}{2}-\frac{1}{u}+\frac{1}{e^{u}-1}$ like in lemma \ref{lem:binet-remainders}. We then have $\lim_{u\rightarrow0}g(u)=0$, $\lim_{u\rightarrow0}g'(u)=\frac{1}{12}$, $\lim_{u\rightarrow0}g''(u)=0$ and $|g'''(u)|\leq\frac{1}{120}$ for $u>0$. Now, if we take equation \eqref{eq:digamma-real}, we obtain
    \begin{equation*}
        -\int_0^\infty\left(\frac{1}{2}-\frac{1}{u}+\frac{1}{e^{u}-1}\right)e^{-us}du=-\int_0^\infty g(u)e^{-us}du=-\frac{1}{12s^2}-\frac{1}{s^3}\int_0^\infty g'''(u)e^{-us}du.
    \end{equation*}
    Hence,
    \begin{equation*}
        \left|R_2(s)\right|\leq\left|\re\left(\int_0^\infty\left(\frac{1}{2}-\frac{1}{u}+\frac{1}{e^{u}-1}\right)e^{-us}du-\frac{1}{12s^2}\right)\right|\leq\frac{1}{120\sigma|s|^3}.
    \end{equation*}
    Lastly, consider the second derivative of Binet's first formula for the logarithmic gamma function. From this we obtain that
    \begin{equation}\label{eq:trigamma-components}
        \begin{split}
            \re(\psi_1(s))=&\frac{\sigma}{\sigma^2+t^2}+\frac{\sigma^2-t^2}{2(\sigma^2+t^2)^2}+\re\left(\int_0^\infty h(u)e^{-us}du\right),\text{ and}\\
            \im(\psi_1(s))=&-\frac{t}{\sigma^2+t^2}-\frac{\sigma t}{(\sigma^2+t^2)^2}+\im\left(\int_0^\infty h(u)e^{-us}du\right),
        \end{split}   
    \end{equation}
    where $h(u)=u\left(\frac{1}{2}-\frac{1}{u}+\frac{1}{e^{u}-1}\right)$ like in lemma \ref{lem:binet-remainders}. Then $\lim_{u\rightarrow0}h(u)=0$, $\lim_{u\rightarrow0}h'(u)=0$, $\lim_{u\rightarrow0}h''(u)=\frac{1}{6}$ and $|h'''(u)|\leq\frac{1}{22}$ for $u>0$. Therefore,
    \begin{equation*}
        \int_0^\infty\left[u\left(\frac{1}{2}-\frac{1}{u}+\frac{1}{e^{u}-1}\right)\right]e^{-us}du=\int_0^\infty h(u)e^{-us}du=\frac{1}{6s^3}+\frac{1}{s^3}\int_0^\infty h'''(u)e^{-us}du.
    \end{equation*}
    And in sight of \eqref{eq:trigamma-components},
    \begin{equation*}
        \begin{split}
            |R_3(s)|\leq\left|\int_0^\infty\left[u\left(\frac{1}{2}-\frac{1}{u}+\frac{1}{e^{u}-1}\right)\right]e^{-us}du\right|\leq&\frac{1}{6|s|^3}+\frac{1}{22\sigma|s|^3}\ \text{ and}\\
            |R_4(s)|\leq\left|\int_0^\infty\left[u\left(\frac{1}{2}-\frac{1}{u}+\frac{1}{e^{u}-1}\right)\right]e^{-us}du\right|\leq&\frac{1}{6|s|^3}+\frac{1}{22\sigma|s|^3}.
        \end{split}
    \end{equation*}
\end{proof}
Now we can estimate the contribution of the gamma factor $\gamma_K(s)$.
\begin{lemma}\label{lem:gamma-factor}
    For $d,a_1,a_2,a_3$ as in Lemma \ref{lem:f-bounds} and $T\geq1$ we obtain
    \begin{equation*}
        \begin{split}
            E_u(\gamma_K)(T)\leq& \frac{n_KT}{\pi}\log\left(\frac{T}{2\pi e}\right)+\frac{n_Kda_1}{4}\log\left(\frac{T}{2\pi}\right)+0.05587n_K\ \text{ and}\\
            E_l(\gamma_K)(T)\geq&\frac{n_KT}{\pi}\log\left(\frac{T}{2\pi e}\right)-\frac{n_Kda_1}{4}\log\left(\frac{T}{2\pi}\right)-0.25n_K.
        \end{split}
    \end{equation*}
\end{lemma}
\begin{proof}
    First, use \eqref{eq:gamma-factor-legendre} to decompose the gamma factor into its $r_1$- and $r_2$-dependent parts. We can use Proposition \ref{prop:gamma-estimates} to estimate $E_u$ and $E_l$ applied to $\gamma_1$ and $\gamma_2$. We can use Proposition \ref{prop:gamma-estimates} to obtain that
    \begin{equation}\label{eq:gamma-term-evaluations}
        \begin{split}
            E_u(\gamma_1)(T)=&\frac{T}{\pi}\log\left(\frac{T}{2\pi e}\right)+\frac{da_1}{4}\log\left(\frac{T}{2\pi}\right)+W_{1,1}(T)+R_{1,1}(T),\\
            E_l(\gamma_1)(T)=&\frac{T}{\pi}\log\left(\frac{T}{2\pi e}\right)-\frac{da_1}{4}\log\left(\frac{T}{2\pi}\right)+W_{1,-1}(T)+R_{1,-1}(T),\\
            E_u(\gamma_2)(T)=&\frac{2T}{\pi}\log\left(\frac{T}{2\pi e}\right)+\frac{da_1}{2}\log\left(\frac{T}{2\pi}\right)+W_{2,1}(T)+R_{2,1}(T),\text{ and}\\
            E_l(\gamma_2)(T)=&\frac{2T}{\pi}\log\left(\frac{T}{2\pi e}\right)-\frac{da_1}{2}\log\left(\frac{T}{2\pi}\right)+W_{2,-1}(T)+R_{2,-1}(T),
        \end{split}
    \end{equation}
    where
    \begin{equation*}
        \begin{split}
            W_{\kappa,\upsilon}(T):=&\kappa\left(\left(\frac{T}{\pi}+\frac{da_1\upsilon}{8}\right)\log\left(1+\left(\frac{1+2d}{2T}\right)^2\right)-\frac{T}{2\pi}\log\left(1+\left(\frac{1+4d}{2T}\right)^2\right)\right)\\
            &+\frac{\kappa(1+2d)-2}{\pi}\arctan\left(\frac{T}{\frac{1}{2}+d}\right)-\frac{\kappa(1+4d)-2}{2\pi}\arctan\left(\frac{T}{\frac{1}{2}+2d}\right)\\
            &-\frac{1}{3\pi\kappa}\left(\frac{2T}{\left(\frac{1}{2}+d\right)^2+T^2}-\frac{T}{\left(\frac{1}{2}+2d\right)^2+T^2}\right)\\
            &-\frac{da_1\upsilon}{4}\left(\frac{\frac{1}{2}+d}{\left(\frac{1}{2}+d\right)^2+T^2}+\frac{\left(\frac{1}{2}+d\right)^2-T^2}{3\kappa\left(\left(\frac{1}{2}+d\right)^2+T^2\right)^2}\right)\\
            &-\frac{d^2a_2\kappa\upsilon}{4}\left(\frac{\frac{1}{2}+d}{\left(\frac{1}{2}+d\right)^2+T^2}+\frac{\left(\frac{1}{2}+d\right)^2-T^2}{\kappa\left(\left(\frac{1}{2}+d\right)^2+T^2\right)^2}\right)\\
            &-\frac{a_3\kappa}{4}\left(\frac{T}{\left(\frac{1}{2}+d\right)^2+T^2}+\frac{(1+2d)T}{\kappa\left(\left(\frac{1}{2}+d\right)^2+T^2\right)^2}\right),
        \end{split}
    \end{equation*}
    and
    \begin{equation*}
        \begin{split}
            R_{\kappa,\upsilon}(T)=&\frac{4}{\pi}R_1\left(\frac{\kappa}{2}s_1\right)-\frac{2}{\pi}R_1\left(\frac{\kappa}{2}s_2\right)+\frac{da_1\upsilon\kappa}{4}R_2\left(\frac{\kappa}{2}s_1\right)\\
            &-\frac{d^2a_2\upsilon\kappa^2}{8}R_3\left(\frac{\kappa}{2}s_1\right)+\frac{a_3\kappa^2}{8}R_4\left(\frac{\kappa}{2}s_1\right),
        \end{split}
    \end{equation*}
    where $s_1=\frac{1}{2}+d+iT$ and $s_2=\frac{1}{2}+2d+iT$. We can then use Proposition \ref{prop:gamma-estimates} to bound the error $R_{\kappa,\upsilon}(T)$ and obtain
    \begin{align*}
        |R_{\kappa,\upsilon}(T)|&\leq\frac{8}{\kappa^{3}}\Biggl[\frac{2}{\pi}\left(\frac{1}{360}+\frac{2}{1011\kappa\re(s_2)}\right)|s_2|^{-3}\\
        &+\left(\frac{4}{\pi}\left(\frac{1}{360}+\frac{2}{1011\kappa\re(s_1)}\right)+\frac{da_1}{240\re(s_1)}+\frac{\kappa^2S}{8}\left(\frac{1}{6}+\frac{2}{22\kappa\re(s_1)}\right)\right)|s_1|^{-3}\Biggr]\\
        &=:UB_\kappa(T),
    \end{align*}
    for $S=|d^2a_2|+|a_3|$, when $T\geq1$. It remains to obtain uniform bounds for
    \begin{equation*}
        F_{\kappa,\upsilon}(T):=W_{\kappa,\upsilon}(T)+\upsilon UB_\kappa(T)
    \end{equation*}
    for $T\geq1$. We prove these bounds by a computer-assisted argument, as above, the code uses the ball arithmetic of Arb~\cite{Johansson2017}. For the two upper bounds, corresponding to $\upsilon=1$, we take $B_{\kappa,1}$ to be $F_{\kappa,1}(1)$, rounded upwards to nine decimal places. For the two lower bounds, corresponding to $\upsilon=-1$, we take
    \begin{equation*}
        B_{1,-1}=\lim_{T\rightarrow\infty}F_{1,-1}=-\frac14\quad\text{and}\quad  B_{2,-1}=\lim_{T\rightarrow\infty}F_{2,-1}(T)=0.
    \end{equation*}
    We then verify
    \begin{equation*}
        F_{\kappa,1}(T)\leq B_{\kappa,1}
        \quad\text{and}\quad
        F_{\kappa,-1}(T)\geq B_{\kappa,-1}
    \end{equation*}
    for all $T\geq1$. We do this by splitting the interval at a height $T_1$. We handle the finite interval $[1,T_1]$ by a method similar to that used in Lemmas \ref{lem:f-bounds} and \ref{lem:polynomial-factor}. The code starts with the full interval $I=[1,T_1]$, then computes the enclosures of $F_{\kappa,\upsilon}$ and $F'_{\kappa,\upsilon}$ on $I$ using automatic differentiation in ball arithmetic. For $\upsilon=1$ the code accepts an interval $I$ if either $F_{\kappa,\upsilon}(I)\subset(-\infty,B_{\kappa,\upsilon})$ or if $0\not\in F'_{\kappa,\upsilon}(I)$ and $F_{\kappa,\upsilon}\leq B_{\kappa,\upsilon}$ on the endpoints of $I$. If neither of these conditions are satisfied then the code bisects the interval $I$ and reruns the test. The lower bounds for $\upsilon=-1$ are handled analogously. Thus, if all boxes are eventually accepted, then the $B_{\kappa,\upsilon}$ inequality holds for the entire interval $[1,T_1]$. To cover the inequalities on the tail $T\geq T_1$ we use the substitution $T=\frac{1}{u}$, so that $0<u\leq T_1^{-1}$, and rewrite the expressions defining $F_{\kappa,\upsilon}(T)$ in terms of $u$. We use the alternating-series estimates
    \begin{equation*}
        x-\frac{x^2}{2}+\frac{x^3}{3}-\frac{x^4}{4}\leq\log(1+x)\leq x-\frac{x^2}{2}+\frac{x^3}{3}
    \end{equation*}
    and
    \begin{equation*}
        y-\frac{y^3}{3}+\frac{y^5}{5}-\frac{y^7}{7}\leq\arctan(y)\leq y-\frac{y^3}{3}+\frac{y^5}{5},
    \end{equation*}
    where in each term we choose the upper or lower estimate according to the sign of its coefficient and the direction of the desired inequality. For the remainder bound, we use
    \begin{equation*}
        |s_1|^{-3}\leq T^{-3}=u^3
        \quad\text{and}\quad
        |s_2|^{-3}\leq T^{-3}=u^3,
    \end{equation*}
    as well a rational enclosure for $\pi$. Using these estimates the code defines rational functions for which it shows positivity for $0<u\leq T_1^{-1}$ which then implies the desired bound for $F_{\kappa,\upsilon}$. After removing the power of $u$ which vanishes at $u=0$, the code verifies exactly, using Sturm's theorem over $\mathbb{Q}$, that the remaining polynomial has positive constant term and no zeros on this interval. This proves the desired inequality for every $T\geq T_1$. The code tries $T_1\in\{10,20,40,80,160\}$ until both the finite-range and tail arguments succeed. Combining the two parts gives
    \begin{align*}
        W_{1,1}(T)+R_{1,1}(T)\leq& W_{1,1}(T)+UB_{1}(T)\leq-0.14931,\\
        W_{2,1}(T)+R_{2,1}(T)\leq& W_{2,1}(T)+UB_{2}(T)\leq0.11173,\\
        W_{1,-1}(T)+R_{1,-1}(T)\geq& W_{1,-1}(T)-UB_{1}(T)\geq-0.25,\text{ and}\\
        W_{2,-1}(T)+R_{2,-1}(T)\geq& W_{2,-1}(T)-UB_{2}(T)\geq0\\
    \end{align*}
    for $T\geq1$. Combining these estimates with equations \eqref{eq:gamma-term-evaluations} and the fact that $n_K=r_1+2r_2$ to obtain that
    \begin{align*}
        E_u(\gamma_K)(T)\leq&\frac{n_KT}{\pi}\log\left(\frac{T}{2\pi e}\right)+\frac{n_Kda_1}{4}\log\left(\frac{T}{2\pi}\right)+0.05587n_K,\text{ and}\\
        E_l(\gamma_K)(T)\geq&\frac{n_KT}{\pi}\log\left(\frac{T}{2\pi e}\right)-\frac{n_Kda_1}{4}\log\left(\frac{T}{2\pi}\right)-0.25n_K,
    \end{align*}
    completing the proof.
\end{proof}
Lastly, we need to estimate the contribution of the zeta terms.
\begin{lemma}\label{lem:zeta-factor}
    For $d,a_1,a_2,a_3$ as in Lemma \ref{lem:f-bounds} and $T\geq1$ we obtain
    \begin{equation*}
        \begin{split}
            E_u(\z)(T)\leq&7.91048n_K \ \text{ and}\\
            E_l(\z)(T)\geq&-7.91048n_K.
        \end{split}
    \end{equation*}
\end{lemma}
\begin{proof}
    Because $\frac{1}{2}+d>1$ and $\frac{1}{2}+2d>1$ we may express $\z(\frac{1}{2}+d),\z(\frac{1}{2}+2d)$ and related functions with the help of the Euler product \eqref{eq:euler-product}. We obtain that
    \begin{equation}\label{eq:zeta-euler-derivatives}
        \begin{split}
            \log\z(s)=&-\sum_{\mathfrak{p}\subset\mathcal{O}_K}\log(1-(N\mathfrak{p})^{-s}),\\
            \frac{\z'}{\z}(s)=&\sum_{\mathfrak{p}\subset\mathcal{O}_K}\frac{\log(N\mathfrak{p})}{1-(N\mathfrak{p})^{s}},\text{ and}\\
            \left(\frac{\z'}{\z}\right)'(s)=&\sum_{\mathfrak{p}\subset\mathcal{O}_K}\frac{(\log(N\mathfrak{p}))^2(N\mathfrak{p})^s}{(1-(N\mathfrak{p})^s)^2}.
        \end{split}
    \end{equation}
    Here, $\log\z(s)$ is again a special function and should be understood as the integral $\int_2^s\frac{\z'}{\z}(z)dz+\log(\z(2))$, rather than any particular branch of the logarithm applied to $\z(s)$. We use the notation $\log\z(s)$ without parentheses to highlight this distinction. We can use these expressions to compute $E_u$ applied to $\z$. We obtain
    \begin{equation}\label{eq:max-q1-ideal-sum}
        \begin{split}
            E_u(\z)(T)=\sum_{\mathfrak{p}\subset\mathcal{O}_K}&q_1(N\mathfrak{p},\varphi)\leq\sum_{\mathfrak{p}\subset\mathcal{O}_K}\max_{\varphi}q_1(N\mathfrak{p},\varphi),
        \end{split}
    \end{equation}
    where $\varphi=T\log(N\mathfrak{p})$ and
    \begin{equation*}
        \begin{split}
            q_1(x,\varphi)=&\frac{2}{\pi}f_1(x,\varphi)+\frac{da_1}{2}f_2(x,\varphi)+\frac{d^2a_2}{2}f_3(x,\varphi)+\frac{a_3}{2}f_4(x,\varphi),\ \text{ with}\\
            f_1(x,\varphi):=&-2\arctan\left(\frac{\sin\varphi}{X-\cos\varphi}\right)+\arctan\left(\frac{\sin\varphi}{Y-\cos\varphi}\right),\\
            f_2(x,\varphi):=&L\frac{1-X\cos\varphi}{1-2X\cos\varphi+X^2},\\
            f_3(x,\varphi):=&-L^2\frac{X\left(\left(1+X^2\right)\cos\varphi-2X\right)}{\left(1-2X\cos\varphi+X^2\right)^2},\text{ and}\\
            f_4(x,\varphi):=&L^2\frac{X\left(1-X^2\right)\sin\varphi}{\left(1-2X\cos\varphi+X^2\right)^2},
        \end{split}
    \end{equation*}
    where $X=x^{\frac{1}{2}+d}$, $Y=x^{\frac{1}{2}+2d}$, and $L=\log(x)$. We will now follow the more direct approach of Bellotti and Wong \cite{BellottiWong2026Artin} and estimate the prime ideal sum by a prime sum where we take the maximum of each term over every power of that prime instead of first showing that this the maximum will be attained at the first power. Next, we need a majorant for $q_1(x,\varphi)$ to handle the tail of the sum in equation \eqref{eq:max-q1-ideal-sum}. We handle the remaining finite sum in Python using interval arithmetic to obtain a certified upper bound. We consider the individual terms of $q_1$ separately. For $f_1$ we will use that the function $h_1(y)=\text{artanh}(y)+\arctan\left(\frac{y\sin(\varphi)}{1-y\cos(\varphi)}\right)$ is non-decreasing on the interval $[0,1)$ with $h_1(0)=0$. To see this, we compute the derivative of $h_1(y)$ by
    \begin{equation*}
        \begin{split}
            h_1'(y)=&\frac{1}{1-y^2}+\frac{\sin(\varphi)}{1-2y\cos(\varphi)+y^2}=\frac{1-2y\cos(\varphi)+y^2+(1-y^2)\sin(\varphi)}{(1-y^2)(1-2y\cos(\varphi)+y^2)}\\
            =&\frac{(1-\sin(\varphi))\left(y-\frac{\cos(\varphi)}{1-\sin(\varphi)}\right)^2}{(1-y^2)(1-2y\cos(\varphi)+y^2)}\geq0.
        \end{split}
    \end{equation*}
    Now, because $1>X^{-1}\geq Y^{-1}\geq0$ we have $0\leq h_1\left(Y^{-1}\right)\leq h_1\left(X^{-1}\right)$ and we can deduce that
    \begin{equation}\label{eq:u1-majorant}
        \begin{split}
            \frac{2}{\pi}f_1(x,\varphi)&=\frac{2}{\pi}\left(2\text{artanh}\left(X^{-1}\right)-\text{artanh}\left(Y^{-1}\right)-\left(2h_1\left(X^{-1}\right)-h_1\left(Y^{-1}\right)\right)\right)\\
            &\leq\frac{4}{\pi}\text{artanh}\left(X^{-1}\right)-\frac{2}{\pi}\text{artanh}\left(Y^{-1}\right)\\
            &=\frac{2}{\pi}\left(\log\left(1+X^{-1}\right)-\log\left(1-X^{-1}\right)\right)-\frac{1}{\pi}\left(\log\left(1+Y^{-1}\right)-\log\left(1-Y^{-1}\right)\right)\\
            &=:u_1(x)
        \end{split}
    \end{equation}
    for $x>1$. For $f_2$ we obtain
    \begin{equation}\label{eq:u2-majorant}
        \begin{split}
            \frac{da_1}{2}f_2(x,\varphi)=&\frac{da_1}{2}L\frac{1-X\cos(\varphi)}{1-2X\cos(\varphi)+X^2}\leq\frac{da_1}{2}\frac{L}{1+X}=:u_2(x)
        \end{split}
    \end{equation}
    when $x>1$. Lastly, we may estimate the remaining terms by
    \begin{multline}\label{eq:u3-majorant}
            \frac{d^2a_2}{2}f_3(x,\varphi)+\frac{a_3}{2}f_4(x,\varphi)=-\frac{d^2a_2}{2}\re\left(\frac{L^2Xe^{i\varphi}}{\left(1-Xe^{i\varphi}\right)^2}\right)+\frac{a_3}{2}\im\left(\frac{L^2Xe^{i\varphi}}{\left(1-Xe^{i\varphi}\right)^2}\right)\\
            \leq c\left|\frac{L^2Xe^{i\varphi}}{\left(1-Xe^{i\varphi}\right)^2}\right|= c\frac{L^2X}{1-2X\cos(\varphi)+X^2}
            \leq c\frac{L^2X}{(1-X)^2}=:u_3(x),
    \end{multline}
    where $c=c_{d,a_2,a_3}:=\max_{s\in\C\setminus\{0\}}\frac{-\frac{d^2a_2}{2}\re(s)+\frac{a_3}{2}\im(s)}{|s|}=\frac{1}{2}\sqrt{(-d^2a_2)^2+a_3^2}$, and $x>1$. Now define the function $g_M(x)$ for some positive cut-off $M$ via
    \begin{equation*}
        g_M(x)=
        \begin{cases}
            u_1(x)+u_2(x)+u_3(x), & \text{if }x>M\text{ and}\\
            \max_\varphi q_1(x,\varphi), & \text{otherwise.}
        \end{cases}
    \end{equation*}
    By taking derivatives, we can see that the functions $u_1$, $u_2$, and $u_3$ are decreasing in $x$ for all $x$ that satisfy $\left(\frac{1}{2}+d\right)\log(x)>2$, choosing $x\geq10$ suffices for all $d>0.5$. We therefore know that $g_M(x)$ is decreasing in $x$ if $x>M\geq10$. Now, define $m(p)=\sup_{f\geq1}g_M\left(p^f\right)$. Notice that the supremum is a maximum because there are only finitely many powers $p^f\leq M$ for any finite $M$ and because $g_M(x)$ is decreasing for $x> M$. Furthermore, we can see that $g_M(x)$ is positive for $x>1$ because both $u_1(x)+u_2(x)+u_3(x)$ and $\max_\varphi q_1(x,\varphi)\geq q_1\left(x,\frac{3\pi}{2}\right)$ are positive for all $x>1$. The second point follows, because all four terms $f_1$, $f_2$, $f_3$, and $f_4$ are positive at $\varphi=\frac{3\pi}{2}$. Hence, $m(p)$ is positive for all prime number $p\in\mathbb{P}$. Thus, we can upper bound equation \eqref{eq:max-q1-ideal-sum} by
    \begin{equation*}
        \begin{split}
            \sum_{\mathfrak{p}\subset\mathcal{O}_K}\max_{\varphi}q_1(N\mathfrak{p},\varphi)\leq&\sum_{\mathfrak{p}\subset\mathcal{O}_K}g_M(N\mathfrak{p})=\sum_{p\in\mathbb{P}}\sum_{\substack{\mathfrak{p}\subset\mathcal{O}_K \\ N\mathfrak{p}=p^f \\ f\in\Z_{\geq1}}}g_M\left(p^f\right)\leq\sum_{p\in\mathbb{P}}\sum_{\substack{\mathfrak{p}\subset\mathcal{O}_K \\ N\mathfrak{p}=p^f \\ f\in\Z_{\geq1}}}m(p)\\
            \leq&n_K\sum_{p\in\mathbb{P}}m(p)=n_K\sum_{\substack{p\in\mathbb{P} \\ p\leq M}}m\left(p\right)+n_K\sum_{\substack{p\in\mathbb{P} \\ p>M}}(u_1(p)+u_2(p)+u_3(p)).
        \end{split}
    \end{equation*}
    Here we used that there can be at most $n_K$ prime ideals $\mathfrak{p}$ in $\mathcal{O}_K$ over any prime number $p\in\mathbb{P}$ \cite[Proposition 8.2]{Neukirch1999}. These are exactly the prime ideals with norm $N\mathfrak{p}$ divisible by $p$. We also used that $m(p)$ is positive and that $m(p)=g_M(p)=u_1(p)+u_2(p)+u_3(p)$ for primes $p>M$ because in this range $g_M(x)$ is decreasing and any sup over prime powers $p^f$ will be attained by $f=1$. We can use equations \eqref{eq:zeta-euler-derivatives}, \eqref{eq:u1-majorant}, \eqref{eq:u2-majorant}, and \eqref{eq:u3-majorant} to estimate the tail sum by
    \begin{equation*}
        \begin{split}
            \sum_{\substack{p\in\mathbb{P}\\ p>M}}(u_1(p)+u_2(p)+u_3(p))
            =&\frac{2}{\pi}\left(2\log\zeta\left(\frac{1}{2}+d\right)-\log\zeta(1+2d)\right)\\
            &-\frac{1}{\pi}\left(2\log\zeta\left(\frac{1}{2}+2d\right)-\log\zeta(1+4d)\right)\\
            &+\frac{da_1}{2}\left(2\frac{\zeta'}{\zeta}(1+2d)-\frac{\zeta'}{\zeta}\left(\frac{1}{2}+d\right)\right)\\
            &+c\left(\frac{\zeta'}{\zeta}\right)'\left(\frac{1}{2}+d\right)\\
            &-\sum_{\substack{p\in\mathbb{P}\\ p\leq M}}(u_1(p)+u_2(p)+u_3(p))\\
            \leq&1.862905,
        \end{split}
    \end{equation*}
    where we chose $M=10^5$. Using interval arithmetic to evaluate the remaining terms in equation \eqref{eq:max-q1-ideal-sum}, we obtain the certified upper bound
    \begin{equation*}
        \sum_{p\in\mathbb{P},p\leq M}m(p)\leq 6.047575.
    \end{equation*}
    Altogether, we obtain that $E_u(\z)(T,d)\leq7.91048n_K$. We obtain the lower bound for $E_l(\z)(T,d)$ from the upper bound by observing that it is exactly the odd part (in $\varphi$) of $E_u(\z)(T,d)$ minus the even part of $E_u(\z)(T,d)$. We have that
    \begin{equation*}
        \begin{split}
            E_l(\z)(T)=\sum_{\mathfrak{p}\subset\mathcal{O}_K}&q_2(N\mathfrak{p},\varphi)\geq\sum_{\mathfrak{p}\subset\mathcal{O}_K}\min_{\varphi}q_2(N\mathfrak{p},\varphi),
        \end{split}
    \end{equation*}
    where 
    \begin{equation*}
        q_2(x,\varphi)=\frac{2}{\pi}f_1(x,\varphi)-\frac{da_1}{2}f_2(x,\varphi)-\frac{d^2a_2}{2}f_3(x,\varphi)+\frac{a_3}{2}f_4(x,\varphi).
    \end{equation*}
    Then
    \begin{equation*}
        \min_\varphi q_2(x,\varphi)=-\max_\varphi(-q_2(x,\varphi))=-\max_\varphi q_1(x,-\varphi)=-\max_\varphi q_1(x,\varphi),
    \end{equation*}
    Hence, the minimum of $q_2(x,\varphi)$ is minus the maximum of $q_1(x,\varphi)$ and we may lower bound $E_l(\z)(T)$ by the negative of the upper bound for $E_u(\z)(T)$. This concludes the proof of the lemma.
\end{proof}
Combining Lemmas \ref{lem:polynomial-factor}, \ref{lem:discriminant-factor}, \ref{lem:gamma-factor} and \ref{lem:zeta-factor} yields the following inequalities. We have
\begin{equation*}
    \begin{split}
        &\begin{split}
            E_u(\xi_K)(T)\leq&\frac{T}{\pi}\log\left(d_K\left(\frac{T}{2\pi e}\right)^{n_K}\right)+\frac{da_1}{4}\log\left(d_K\left(\frac{T}{2\pi}\right)^{n_K}\right)\\
        &+7.96635n_K+2.00208
        \end{split}
        \quad\text{ and}\\
        &\begin{split}
            E_l(\xi_K)(T)\geq&\frac{T}{\pi}\log\left(d_K\left(\frac{T}{2\pi e}\right)^{n_K}\right)-\frac{da_1}{4}\log\left(d_K\left(\frac{T}{2\pi}\right)^{n_K}\right)\\
        &-8.16048n_K+1.83865
        \end{split}
    \end{split}
\end{equation*}
for $T\geq1$. Plugging in $d=0.722$ and $a_1=1.07$ then yields
\begin{equation*}
    \begin{split}
        E_u(\xi_K)&(T)\leq\frac{T}{\pi}\log\left(d_K\left(\frac{T}{2\pi e}\right)^{n_K}\right)+\frac{da_1}{4}(\log d_K +n_K\log T)\\
        &+\left(7.96635-\frac{da_1}{4}\log(2\pi)\right)n_K+2.00208\\
        \leq&\frac{T}{\pi}\log\left(d_K\left(\frac{T}{2\pi e}\right)^{n_K}\right)+0.194(\log d_K+n_K\log T)+7.61140n_K+2.00208,\\
        E_l(\xi_K)&(T)\geq\frac{T}{\pi}\log\left(d_K\left(\frac{T}{2\pi e}\right)^{n_K}\right)-\frac{da_1}{4}(\log d_K +n_K\log T)\\
        &-\left(8.16048-\frac{da_1}{4}\log(2\pi)\right)n_K+1.83865\\
        \geq&\frac{T}{\pi}\log\left(d_K\left(\frac{T}{2\pi e}\right)^{n_K}\right)-0.194(\log d_K+n_K\log T)-7.80553n_K+1.83865
    \end{split}
\end{equation*}
Hence,

\begin{equation*}
    \begin{split}
        \left|N_K(T)-\frac{T}{\pi}\log\left(d_K\left(\frac{T}{2\pi e}\right)^{n_K}\right)\right|
        \leq&0.194(\log d_K+n_K\log T)+7.806n_K+1.808
    \end{split}
\end{equation*}
for $T\geq1$, where $T$ satisfies $\z(x+iT)\neq0$ for all $x\in\R$. Here we use the fact that $n_K\geq1$. Theorem \ref{thm:main} follows from this and equation \eqref{eq:half-weight-continuation}.
\begin{remark}
    Table \ref{tab:dedekind-tuples} contains different choices for $C_1,C_2,C_3$ such that 
    \begin{equation}\label{eq:NK-bound}
        \left|N_K(T)-\frac{T}{\pi}\log\left(d_K\left(\frac{T}{2\pi e}\right)^{n_K}\right)\right|\leq C_1(\log d_K +n_K\log T)+C_2n_K+C_3
    \end{equation}
    for the different tuples $(d,a_1,a_2,a_3)$ from table \ref{tab:parameters}.
    \begin{table}[h!]
    \centering
    \begin{tabular}{|c|c|c|c|}
    \hline
    $j$ & $C_1$ & $C_2$ & $C_3$ \\ [0.5ex] 
    \hline
    1 & $0.2$ & $6.436$ & $1.843$\\
    2 & $0.24$ & $3.714$ & $1.945$\\
    3 & $0.28$ & $2.540$ & $2.072$\\
    4 & $0.32$ & $1.859$ & $2.194$\\
    5 & $0.36$ & $1.399$ & $2.295$\\
    6 & $0.4$ & $1.057$ & $2.387$\\
    [1ex]  
    \hline
    \end{tabular}
    \caption{Further admissible choices for $(C_1,C_2,C_3)$ in equation \eqref{eq:NK-bound}}
    \label{tab:dedekind-tuples}
    \end{table}\\
    We find that these improve on the results in table \ref{tab:previous-dedekind}\cite{HasanalizadeShenWong2022Dedekind}. Table \ref{tab:dedekind-comparison} compares the bounds from this method to the bounds in \cite{HasanalizadeShenWong2022Dedekind} at the example of the Dedekind zeta-function associated with the field $K=\mathbb{Q}(i)$ with $n_K=2$ and absolute discriminant $d_K=4$.
    \begin{table}[h!]
    \centering
    \begin{tabular}{|c|c|c|c|}
    \hline
    $T$ & table \ref{tab:previous-dedekind} \cite{HasanalizadeShenWong2022Dedekind} & table \ref{tab:dedekind-tuples} & achieved by tuple $j$\\ [0.5ex] 
    \hline
    $10$ & $12.485$ & $6.898$ & $6$\\
    $10^3$ & $15.210$ & $10.566$ & $5$\\
    $10^6$ & $19.298$ & $15.198$ & $4$\\
    $10^9$ & $23.386$ & $19.146$ & $3$\\
    $10^{12}$ & $27.474$ & $22.969$ & $2$\\
    $10^{15}$ & $31.488$ & $26.285$ & $2$\\
    [1ex]  
    \hline
    \end{tabular}
    \caption{Comparison of the results from this paper with the results from \cite{HasanalizadeShenWong2022Dedekind}}
    \label{tab:dedekind-comparison}
    \end{table}
\end{remark}
\begin{proof}[Proof of Corollary 1.2]
    We apply Theorem \ref{thm:main} in the case where $K=\mathbb{Q}$, yielding that
    \begin{equation*}
        \left|N_{\mathbb{Q}}(T)-\frac{T}{\pi}\log\left(\frac{T}{2\pi e}\right)\right|\leq0.194\log T+9.614.
    \end{equation*}
    Now, we observe that due to the complex symmetry of the zeros of $\zeta(s)$ and because $\zeta$ has no zeros on the real line inside the critical strip, $N_{\mathbb{Q}}(T)=2N(T)$ for all $T\geq0$. Hence,
    \begin{equation}\label{eq:riemann-asymptotic}
        \left|N(T)-\frac{T}{2\pi}\log\left(\frac{T}{2\pi e}\right)\right|\leq0.097\log T+4.807
    \end{equation}
    for $T\geq1$, as desired.
\end{proof}
\begin{remark}
    Table \ref{tab:riemann-tuples} contains different choices for $C_1,C_2$ such that 
    \begin{equation}\label{eq:riemann-bound}
        \left|N(T)-\frac{T}{2\pi}\log\left(\frac{T}{2\pi e}\right)\right|\leq C_1\log(T)+C_2
    \end{equation}
    for the different tuples $(d,a_1,a_2,a_3)$ from table \ref{tab:parameters}.
    \begin{table}[h!]
    \centering
    \begin{tabular}{|c|c|c|c|}
    \hline
    $j$ & $C_1$ & $C_2$ & optimal range\\ [0.5ex] 
    \hline
    1 & $0.1$ & $4.140$ & $\left(e^{65.5},e^{222.\overline{3}}\right)$\\
    2 & $0.12$ & $2.830$ & $\left(e^{26.2},e^{65.5}\right)$\\
    3 & $0.14$ & $2.306$ & $\left(e^{13.95},e^{26.2}\right)$\\
    4 & $0.16$ & $2.027$ & $\left(e^{9},e^{13.95}\right)$\\
    5 & $0.18$ & $1.847$ & $\left(e^{6.25},e^{9}\right)$\\
    6 & $0.2$ & $1.722$ & $\left(1,e^{6.25}\right)$ \\
    [1ex]  
    \hline
    \end{tabular}
    \caption{Admissible choices for $(C_1,C_2)$ for equation \eqref{eq:riemann-bound} with the corresponding ranges where each estimate is best.}
    \label{tab:riemann-tuples}
    \end{table}\\
    Beyond $e^{222.\overline{3}}$, the bound in equation \eqref{eq:riemann-asymptotic} gives the best result from this method.
\end{remark}
\begin{remark}
    Notice that this method to estimate the error in the zero counting function of a $L$-function predominantly depends on Lemma \ref{lem:f-bounds}. However, to apply this lemma, we simply require that $\frac{1}{2}+iT-\rho$ is within the critical strip. Hence, the method is in principle amenable to any $L$-function in the Selberg class as long as one also has sufficient estimates of the gamma factor.
\end{remark}

\section{Methods}

The computer-assisted estimates in this article were carried out in Python using Arb ball arithmetic through the \texttt{python-flint} package. The accompanying script \texttt{certify\_paper.py} certifies the parameter inequalities and numerical bounds used in the proof and computes the resulting constants in Theorem~\ref{thm:main} and Corollary~\ref{cor:riemann}. The script requires \texttt{python-flint} and \texttt{sympy} and is included with the ancillary files of the arXiv submission \url{https://arxiv.org/abs/2510.27444}. Instructions for reproducing the computations are provided at the beginning of the script. \\ 
The author used ChatGPT (OpenAI) and Claude (Anthropic), with access through their affiliation with Chalmers University of Technology and the University of Gothenburg, during the preparation of this manuscript to assist with proofreading, improving the clarity of the exposition, and especially the accompanying computational code. All mathematical arguments, computations, references, and AI-assisted suggestions were independently reviewed and verified by the author, who takes full responsibility for the content of the manuscript.
\section*{Acknowledgements}
The author thanks Valentin Blomer for his guidance and suggestions, which helped simplify the argument, and the referee for a careful reading of the manuscript and for pointing out several errors in an earlier draft. The author also thanks Chiara Bellotti and Peng-Jie Wong for pointing out the sign inconsistency involving the $a_3$-term in an earlier version.
\bibliography{sn-bibliography}

@article{HasanalizadeShenWong2022Dedekind,
  author  = {Elchin Hasanalizade and Quanli Shen and Peng-Jie Wong},
  title   = {Counting zeros of {Dedekind} zeta functions},
  journal = {Math. Comp.},
  volume  = {91},
  number  = {333},
  pages   = {277--293},
  year    = {2022}
}

@article{Rosser1941,
  author  = {J. Barkley Rosser},
  title   = {Explicit bounds for some functions of prime numbers},
  journal = {Amer. J. Math.},
  volume  = {63},
  number  = {1},
  pages   = {211--232},
  year    = {1941}
}

@article{KadiriNg2012,
  author  = {Habiba Kadiri and Nathan Ng},
  title   = {Explicit zero density theorems for {Dedekind} zeta functions},
  journal = {J. Number Theory},
  volume  = {132},
  number  = {4},
  pages   = {748--775},
  year    = {2012}
}

@article{Turing1953,
  author  = {A. M. Turing},
  title   = {Some calculations of the {Riemann} zeta-function},
  journal = {Proc. Lond. Math. Soc.},
  series  = {3},
  volume  = {3},
  number  = {1},
  pages   = {99--117},
  year    = {1953}
}

@article{McCurley1984,
  author  = {Kevin S. McCurley},
  title   = {Explicit estimates for the error term in the prime number theorem for arithmetic progressions},
  journal = {Math. Comp.},
  volume  = {42},
  number  = {165},
  pages   = {265--285},
  year    = {1984}
}

@article{Backlund1916,
  author  = {R. J. Backlund},
  title   = {{\"U}ber die Nullstellen der {Riemannschen} Zetafunktion},
  journal = {Acta Math.},
  volume  = {41},
  number  = {1},
  pages   = {345--375},
  year    = {1916}
}

@book{Neukirch1999,
  author     = {J{\"u}rgen Neukirch},
  translator = {Norbert Schappacher},
  title      = {Algebraic Number Theory},
  publisher  = {Springer},
  address    = {Berlin},
  year       = {1999}
}

@book{FischerLieb1994,
  author    = {Wolfgang Fischer and Ingo Lieb},
  title     = {Funktionentheorie},
  edition   = {5},
  publisher = {Vieweg},
  address   = {Braunschweig},
  year      = {1994}
}

@book{WhittakerWatson1920,
  author    = {E. T. Whittaker and G. N. Watson},
  title     = {A Course of Modern Analysis},
  edition   = {3},
  publisher = {Cambridge University Press},
  address   = {Cambridge},
  year      = {1920}
}

@article{BellottiWong2025,
  author  = {Chiara Bellotti and Peng-Jie Wong},
  title   = {Improved estimates for the argument and zero-counting function of the {Riemann} zeta-function},
  journal = {Math. Comp.},
  year    = {2025}
}

@book{IwaniecKowalski2004,
  author    = {Henryk Iwaniec and Emmanuel Kowalski},
  title     = {Analytic Number Theory},
  series    = {American Mathematical Society Colloquium Publications},
  volume    = {53},
  publisher = {American Mathematical Society},
  address   = {Providence, RI},
  year      = {2004}
}

@article{Trudgian2012,
  author  = {Timothy S. Trudgian},
  title   = {An improved upper bound for the argument of the {Riemann} zeta-function on the critical line},
  journal = {Math. Comp.},
  volume  = {81},
  number  = {278},
  pages   = {1053--1061},
  year    = {2012}
}

@article{Trudgian2014,
  author  = {Timothy S. Trudgian},
  title   = {An improved upper bound for the argument of the {Riemann} zeta-function on the critical line {II}},
  journal = {J. Number Theory},
  volume  = {134},
  pages   = {280--292},
  year    = {2014}
}

@article{vonMangoldt1905,
  author  = {H. von Mangoldt},
  title   = {Zur Verteilung der Nullstellen der {Riemannschen} Funktion {$\xi(t)$}},
  journal = {Math. Ann.},
  volume  = {60},
  number  = {1},
  pages   = {1--19},
  year    = {1905}
}

@phdthesis{Grossmann1913,
  author = {J. Grossmann},
  title  = {{\"U}ber die Nullstellen der {Riemannschen} Zeta-Funktion und der {Dirichletschen} {$L$}-Funktionen},
  school = {Georg-August-Universit{\"a}t G{\"o}ttingen},
  year   = {1913}
}

@article{PlattTrudgian2015,
  author  = {David J. Platt and Timothy S. Trudgian},
  title   = {An improved explicit bound on {$|\zeta(1/2+it)|$}},
  journal = {J. Number Theory},
  volume  = {147},
  pages   = {842--851},
  year    = {2015}
}

@article{HasanalizadeShenWong2022Riemann,
  author  = {Elchin Hasanalizade and Quanli Shen and Peng-Jie Wong},
  title   = {Counting zeros of the {Riemann} zeta function},
  journal = {J. Number Theory},
  volume  = {235},
  pages   = {219--241},
  year    = {2022}
}

@misc{BellottiWong2026Artin,
  author       = {Chiara Bellotti and Peng-Jie Wong},
  title        = {Counting zeros of {Artin} {$L$}-functions},
  howpublished = {arXiv:2606.31441},
  year         = {2026}
}

@article{Johansson2017,
  author  = {Fredrik Johansson},
  title   = {{Arb}: efficient arbitrary-precision midpoint-radius interval arithmetic},
  journal = {IEEE Trans. Comput.},
  volume  = {66},
  number  = {8},
  pages   = {1281--1292},
  year    = {2017}
}
\vspace{25pt}
\noindent
\textit{Author email address:} \texttt{victor.amberger@chalmers.se}
\end{document}